\documentclass[10pt, a4paper, reqno]{amsart}
\usepackage[utf8]{inputenc}

\usepackage{bbm,amsmath,amsthm,mathtools,booktabs,bm,color,setspace,url,mathrsfs}
\usepackage{tikz,tikz-3dplot}

\usepackage[
  pdftex,
  plainpages=false,
  pdfpagelabels,
  colorlinks,
  citecolor=black,
  linkcolor=black,
  urlcolor=black,
  filecolor=black,
  bookmarksopen
]{hyperref} 

\usepackage{thmtools, thm-restate}
\declaretheorem[numberwithin=section]{theorem}
\declaretheorem[sibling=theorem]{corollary}
\declaretheorem[sibling=theorem]{lemma}

\theoremstyle{definition}
\declaretheorem[sibling=theorem]{example}
\theoremstyle{remark}
\declaretheorem[sibling=theorem]{remark}

\newcommand{\C}{\mathbb{C}}
\newcommand{\R}{\mathbb{R}}

\newcommand{\Z}{\mathbb{Z}}
\newcommand{\Bcal}{\mathcal{B}}
\newcommand{\Ecal}{\mathcal{E}}
\newcommand{\Fcal}{\mathcal{F}}

\newcommand{\Rcal}{\mathcal{R}}
\newcommand{\fcone}{\mathrm{fcone}}
\newcommand{\tcone}{\mathrm{tcone}}
\newcommand{\aff}{\mathrm{aff}}
\newcommand{\spann}{\mathrm{span}}
\newcommand{\innt}{\mathrm{int}}
\newcommand{\lin}{\mathrm{lin}}
\newcommand{\conv}{\mathrm{conv}}
\newcommand{\vol}{\mathrm{vol}}
\newcommand{\Proj}{\mathrm{Proj}}

\newcommand{\innerp}[2]{\left \langle #1 , #2\right \rangle}
\newcommand{\lp}{\left (}
\newcommand{\rp}{\right )}
\newcommand{\trans}{{\sf T}}
\newcommand{\Bd}{{\mathsf{B}_d}}
\newcommand*\diff{\mathop{}\!\mathrm{d}}

\begin{document}

\begin{abstract}
Macdonald studied a discrete volume measure for a rational polytope $P$, called solid angle sum,  that gives a natural discrete volume for $P$. We give  a local formula for the codimension two quasi-coefficient of the solid angle sum of $P$.  
We also show how to recover the classical Ehrhart quasi-polynomial from the solid angle sum and in particular we find a similar local formula for the codimension one and codimension two quasi-coefficients. These local formulas are naturally valid for all positive real dilates of $P$.   

An interesting open question is to determine necessary and sufficient conditions on a polytope $P$ for which the discrete volume of $P$ given by the solid angle sum equals its continuous volume: $A_P(t) = \vol(P) t^d$.   We prove that a sufficient condition is that  $P$ tiles $\R^d$ by translations, together with the Hyperoctahedral group.
\end{abstract}

\keywords{Lattice points, solid angle, Poisson summation, Fourier transform, polytopes, Bernoulli polynomial, discrete volume, Ehrhart polynomial}

\subjclass[2010]{Primary: 52C07; Secondary: 26B20, 52B20, 52C22}

\thanks{F.C.M was supported by grant \#2017/25237-4, from the S\~ao Paulo Research Foundation (FAPESP). This work was partially supported by Conselho Nacional de Desenvolvimento Cient\'\i fico e Tecnol\'{o}gico --- CNPq (Proc. 423833/2018-9) and by the Coordena\c{c}\~ao de Aperfei\c{c}oamento de Pessoal de N\'\i vel Superior --- Brasil (CAPES) --- Finance Code~001.}

\title{Coefficients of the solid angle and Ehrhart quasi-polynomials}

\author{Fabr\'\i cio Caluza Machado}
\address{F.C. Machado, Instituto de Matem\'atica e Estat\'\i stica, Universidade de S\~ao Paulo\\ Rua do Mat\~ao 1010, 05508-090 S\~ao Paulo/SP, Brazil.}
\email{fabcm1@gmail.com}

\author{Sinai Robins}
\address{S. Robins, Instituto de Matem\'atica e Estat\'\i stica, Universidade de S\~ao Paulo\\ Rua do Mat\~ao 1010, 05508-090 S\~ao Paulo/SP, Brazil.}
\email{sinai.robins@gmail.com}  

\date{January 2, 2022}

\maketitle

\setcounter{tocdepth}{1}
\tableofcontents
\markboth{F.C. Machado, S. Robins}
{Coefficients of the solid angle and Ehrhart quasi-polynomials}

\section{Introduction}\label{sec:intro}

Given a polytope $P \subset \R^d$, the number of integer points within $P$ can be regarded as a discrete analog of the volume of the body. For a \textbf{rational polytope}, meaning that the vertices of $P$ have rational coordinates, Ehrhart~\cite{ehrhart62} showed that the number of integer points in the integer dilates $tP := \{tx : x \in P\}$ can be written as a \textbf{quasi-polynomial} function of~$t$, that is, as an expression of the form
\begin{equation}\label{eq:ehrhart-quasi}
L_P(t) := |tP \cap \Z^d| = \vol(P) t^d + e_{d-1}(t)t^{d-1} + \dots + e_0(t),
\end{equation}
for $t \in \Z$, $t > 0$.  Here, each \textbf{quasi-coefficient} $e_k(t)$ is a periodic function with period dividing the \textbf{denominator} of $P$, defined to be the smallest integer $m$ such that $mP$ in an integer polytope. The function $L_P(t)$ is called the \textbf{Ehrhart quasi-polynomial} of $P$ (see e.g., Beck and Robins~\cite{beck15}). Quasipolynomial behavior appears in many different contexts, and appears for example in the recent work of Bogart, Goodrick, and Woods~\cite{Bogart17}.

The Ehrhart quasi-polynomial of $P$ is not, however, the only discrete volume that we may define. It has a sister polynomial, which is another measure of discrete volume for polytopes. Namely, each integer point located on the boundary of the polytope is assigned a fractional weight, according to the proportion of the space around that point which the polytope occupies. Indeed, Ehrhart and Macdonald already defined this other discrete volume of $P$, calling it the solid angle sum, and we will adopt their notation, as follows. 

At each point $x \in \R^d$, we define the \textbf{solid angle} with respect to $P$:
\begin{equation}\label{eq:angle-def}
\omega_P(x) := \lim_{\epsilon \to 0^+} \frac{\vol(B(x,\epsilon)\cap P)}{\vol(B(x,\epsilon))},
\end{equation}
where $B(x,\epsilon)$ denotes the ball centered at $x$ with radius~$\epsilon$. 
Similarly to Ehrhart, Macdonald~\cite{mcdonald63,mcdonald71} showed that if $P$ is a rational polytope and $t$ is a positive integer, the sum of these fractionally-weighted integer points inside $tP$ is a quasi-polynomial of $t$. We define the \textbf{solid angle sum}
\begin{equation} \label{SolidAngleSum1}
A_P(t) := \sum_{x \in \Z^d} \omega_{tP}(x) =  \vol(P) t^d + a_{d-1}(t)t^{d-1} + \dots + a_0(t),
\end{equation}
and similarly to \eqref{eq:ehrhart-quasi}, we call $a_k(t)$ the quasi-coefficients of $A_P(t)$.

One of the motivations for studying these coefficients is that they capture geometric information about the polytope. Denote by $\vol^*(F)$ the \textbf{relative volume} of a face $F$, which differs from the usual volume inherited from $\R^d$ by a scaling factor such that the fundamental domain of the lattice of integer points on the linear space parallel to the face has volume~$1$. Assuming that $P$ is full-dimensional, it is an easy fact $e_d$ is the volume of~$P$ and, if we further assume that~$P$ is an integer polytope, then it is also fairly easy to show that $e_{d-1}$ is half the sum of the relative volumes of the facets of $P$, and $e_0 = 1$  (see \cite{beck15}). Analogous ``simple'' geometric interpretations for the other coefficients $e_k$ are not yet known. On the other hand, one strong advantage that the solid angle sum has over the Ehrhart polynomial is that it is a better approximation to the volume of $tP$, in the following sense. For a full-dimensional integer polytope $P \subset \R^d$, restricting attention to integer dilates $t$ gives:
\begin{equation}
A_P(t) := \sum_{x \in \Z^d} \omega_{tP}(x) =  \vol(P) t^d + a_{d-2}t^{d-2} + a_{d-4}t^{d-4} +  \dots,
\end{equation}
a polynomial function of $t$, which is an even polynomial in even dimensions, and an odd polynomial in odd dimensions, and also $a_0 = 0$. This was already proved by Macdonald \cite{mcdonald71}, using the purely combinatorial technique of the M\"obius $\mu$-function of the face poset of $P$. 

In this paper, our main focus is on the coefficients of the solid angle quasi-polynomial, as in equation \eqref{SolidAngleSum1}. 
One strong advantage that these quasi-polynomials have over their Ehrhart quasi-polynomial siblings is that 
the solid angle quasi-polynomials are a {\bf simple valuation} on the polytope algebra. This means that for any  given two rational polytopes $P, Q \subset R^d$ whose interiors are disjoint, we have 
$A_{P \cup Q}(t) = A_{P}(t) + A_{Q}(t)$, hence we never have to compute these valuations over intersections of such polytopes.   However, for the Ehrhart polynomials, we have 
$L_{P \cup Q}(t) = L_{P}(t) + L_{Q}(t) - L_{P \cap Q}(t)$, so that in principle one has to compute these latter valuations over lower-dimensional intersections.

To state the main results of the literature, as well as our results here, we need to use the following definitions and data, associated to any polytope $P$. Given any face $F \subseteq P$, we define the \textbf{affine hull} $\aff(F)$ as the smallest affine space containing~$F$ and $\lin(F)$ as the linear subspace parallel to $\aff(F)$.  We also define the \textbf{cone of feasible directions} of $P$ at $F$ as
\[
\fcone(P,F) := \{\lambda (y-x) : x \in F, y \in P, \lambda \geq 0\}
\]
and, picking any point $x_F$ in the relative interior of the face $F$, we define the \textbf{tangent cone} of $P$ at $F$
\[
 \tcone(P,F) := x_F + \fcone(P,F),
\]
as the cone of feasible directions translated to its original position.

McMullen~\cite{mcmullen78} (see also Barvinok~\cite[Chapter 20]{barvinok08}) proved the existence of functions $\mu$ such that for rational~$P$,
\begin{equation}\label{eq:localP}
|P \cap \Z^d| = \sum_{F \subseteq P}\vol^*(F) \mu(P,F),
\end{equation}
where the sum is taken over all faces of $P$ and $\mu$ depends only on ``local'' geometric data associated to the face $F$, namely the cone $\fcone(P,F)$ and the translation class of $\aff(F)$ modulo $\Z^d$. Since the volume is homogeneous with degree $\dim(F)$, applying \eqref{eq:localP} to $tP$ for integer $t$, we see that this expression implies a formula of the type
\begin{equation}\label{eq:localformula}
e_k(t) = \sum_{\substack{F \subset P\\ \dim(F) = k}} \vol^*(F) \mu(tP,tF).
\end{equation}
Such formula is called a \textbf{local formula} for the quasi-coefficients. Since $\fcone(P,F)$ does not change under dilations and, taking $m$ as the denominator of $P$, $\aff(mF)$ has integer points, we see that indeed $e_k(t) = e_k(t+m)$.

These formulas \eqref{eq:localformula} are not unique. Indeed, when $P$ is an integer polytope, Pommersheim and Thomas~\cite{pommersheim04} constructed infinite classes of such formulas based on an expression for the Todd class of a toric variety; 
For the case that $P$ is a rational polytope, Barvinok~\cite{barvinok06,barvinok08} studied the algorithmic complexity of computing these coefficients, showing that fixing the codimension $\dim(P)-\dim(F)$, $\mu(P,F)$ is indeed computable in polynomial time and Berline and Vergne~\cite{berline07} computed a local formula based on a valuation that associates an analytic function to the tangent cone at each face. Garoufalidis and Pommersheim~\cite{garoufalidis12} showed that there exists such valuation (and hence a local formula) uniquely for each given ``rigid complement map'' of the vector space, which is a systematic way to extend functions initialy defined on subspaces to the entire space. Recently, Ring and Sch\"urmann~\cite{ring18} also produced a method to build local formulas based on the choice of fundamental domains on sublattices.   For simplicity, in this paper we assume 
a fixed inner product on $\R^d$, which we also use to identify the space with its dual, and in this way these complement maps are simply given by orthogonal projection.

A simple way to see that the solid angle sum is indeed a quasi-polynomial and enjoys a lot of the same properties of the Ehrhart function follows by using a simple relation~\cite[Lemma 13.2]{beck15} followed by the Ehrhart reciprocity law~\cite[Theorem 4.1]{beck15}:  
\begin{equation}\label{eq:AP-faces}
A_P(t) = \sum_{F \subseteq P}\omega_P(F)L_{\innt(F)}(t) = \sum_{F \subseteq P}\omega_P(F)(-1)^{\dim(F)}L_{F}(-t),
\end{equation}
where the sum is taken over all faces of $P$ and $\omega_P(F)$ is defined as the solid angle of any point in the relative interior of the face $F$.
But we proceed in the opposite direction: first we give formulations for the solid angle quasi-polynomial using Fourier analytic methods and then show how the Ehrhart coefficients can be recovered from them.

We make one more remark concerning the domain of the dilation parameter.  Linke~\cite{linke11} has shown that the Ehrhart function still preserves its quasi-polynomial structure when considered with positive real dilations instead only integer dilations.  One of her main observations was that for a rational polytope $P$ and $p, q \in \Z_{> 0}$, one may use $L_P(p/q) = L_{\frac{1}{q}P}(p)$ and this relation indeed extends to the quasi-coefficients. Letting $e_k(P; t) := e_{k}(t)$, Linke showed that 
\begin{equation}\label{eq:rational-dilation}
e_k(P; p/q) = e_k\big((1/q)P; p\big)q^k. 
\end{equation}
Assuming further that $P$ is full-dimensional, Linke showed  that the quasi-coefficients $e_{k}(P; t)$ are piecewise polynomials of degree $d-k$ with discontinuities only at rational points, which makes the extension to real dilates straightforward.

Taking this observation into account together with the fact that our methods enable the consideration of real dilations quite naturally, we state our results for all positive real dilations.  We do note, however, that as long as we retrict attention to the class of \emph{all} rational polytopes, the main content of the theorems relies only on the integer dilations due to the reduction~\eqref{eq:rational-dilation} above.   

A subtle but important difference occurs when one fixes a single polytope $P$ and compares its Ehrhart function $L_P(t)$ for integer versus real dilations,  the latter carrying  much more information.  In the case where integer translations $P+w$ are considered, we note that the invariance $L_{P+w}(t) = L_P(t)$ is only guaranteed for integer dilations. Recently, Royer has carried out a detailed and extended study of precisely such an analysis. (Royer~\cite{royer17a,royer17b}).

\bigskip
The paper is organized as follows.
In Section~\ref{sec:main.results} we  state our main results. Section~\ref{sec:pre} contains the statement of two lemmas about lattices and a summary of the main results from Diaz, Le and Robins~\cite{diaz16}. Section~\ref{sec:gl-sums} has a formula for a lattice sum that is very useful in the following section and might be of independent interest as well. Section~\ref{sec:ad2} has a proof of the longest theorem of this paper, a local formula for the quasi-coefficient $a_{d-2}(t)$. Section~\ref{sec:Ehrhart} shows how the formula for the solid angle sum quasi-coefficients can be used to determine the Ehrhart quasi-coefficients and we use this to obtain formulas for $e_{d-1}(t)$ and $e_{d-2}(t)$. Section~\ref{sec:examples} has examples of applications of these formulas to some three dimensional polytopes.   

Finally, in Section~\ref{sec:concrete} we define some  interesting families of polytopes called `concrete polytopes', for which the solid angle sum is trivial, in the sense that   $A_P(t) = \vol(P)t^d$ for all positive integers $t$.  We prove that a sufficient condition for such a phenomenon is that the polytope tiles Euclidean space by the Hyperoctahedral group, together with translations.  It is still an open question to determine necessary and sufficient conditions for the occurrence of concrete polytopes.

\bigskip
\section{Main results}\label{sec:main.results}

Our first main result is a local formula for the codimension two quasi-coefficient $a_{d-2}(t)$ of the solid angle sum $A_P(t)$ of a rational polytope $P$, which is then used to derive a similar formula for the codimension two quasi-coefficient $e_{d-2}(t)$ of the Ehrhart quasi-polynomial $L_P(t)$. We also show how these formulas simplify in the case of an integer polytope and integer dilations. 

The statements require the definition of some parameters to describe the tangent cone of each face, its position with respect to the lattice of integer points, and we use some auxiliary functions, namely the Bernoulli polynomials and Dedekind-Rademachersums. We define them first.

\subsection{Definition of the local parameters and the Bernoulli polynomials}

We begin defining some local parameters at each face of $P$, which appear in the statements of our results. Let $P$ be a $d$-dimensional rational polytope in $\R^d$. For each face $F$ of $P$, let $\Lambda_F$ be the lattice of integer vectors orthogonal to $\lin(F)$,
\[\Lambda_F := \lin(F)^{\perp} \cap \Z^d.\]

\begin{figure}[t]
\tdplotsetmaincoords{88}{11}
\begin{tikzpicture}[tdplot_main_coords]

\coordinate (a) at (1,-7);
\coordinate (b) at (1,4);
\coordinate (c) at (-1,4);
\coordinate (d) at (-1,-7);

\begin{scope}[yshift=2.5cm]
\coordinate (e) at (0,-7);
\coordinate (ef) at (0,2);
\coordinate (f) at (0,4);
\end{scope}
\begin{scope}[yshift=2.9cm]
\coordinate (g) at (1,2);
\end{scope}
\begin{scope}[yshift=2.1cm]
\coordinate (h) at (1,2);
\end{scope}

\coordinate (cd) at (intersection cs:
first line={(c) -- (d)}, second line={(e) -- (a)});
                 
\fill[gray, opacity = .6] (c) -- (d) -- (e) -- (f) -- cycle;
\fill[gray, opacity = .4] (a) -- (b) -- (f) -- (e) -- cycle;

\draw[thick] (e) -- (a) -- (b) -- (f) -- (e) -- (d) -- (cd);
\draw[thick, dashed] (cd) -- (c) -- (f);

\draw[->, thick] (ef) -- (g);
\draw[->, thick] (ef) -- (h);

\node[above right = 3] at (a) {\small $F_1$};
\node[above right = 3] at (d) {\small $F_2$};
\node[right] at (g) {\small $N_{P}(F_1)$};
\node[right] at (h) {\small $-N_{P}(F_2)$};
\end{tikzpicture}
\caption{The normal vectors of two facets, used in the computation of $c_G$.}\label{fig:cG}
\end{figure}

If $F$ is a face of $P$ and $G$ is a facet of~$F$, denote by $N_F(G)$ the unit normal vector in $\lin(F)$ pointing outward to $G$. For a ($d-2$)-dimensional face $G$ of $P$, let $F_1 = F_1(G)$ and $F_2 = F_2(G)$ be the two facets whose intersection defines $G$. The solid angle of $G$, also called the \textbf{dihedral angle} of the edge when $d = 3$, can be computed as the angle between the normal vectors $N_P(F_1)$ and $-N_P(F_2)$ (see Figure~\ref{fig:cG}). We let~$c_G$ denotes the cosine of this angle,
\[c_G := -\innerp{N_P(F_1)}{N_P(F_2)},\]
so that $\omega_P(G) = \arccos(c_G)/(2\pi)$. Let $v_{F_1}$, $v_{F_2}$ be the primitive integer vectors in the directions of $N_{P}(F_1)$ and $N_{P}(F_2)$ and let $v_{F_1, G}$, $v_{F_2, G}$ be the $\Lambda_G^*$-primitive vectors in the directions of $N_{F_1}(G)$ and $N_{F_2}(G)$ ($\Lambda_G^*$ stands for the dual lattice, see definition in Section~\ref{sec:lattices}). Let $\bar{x}_G$ be the projection of $G$ onto $\lin(G)^\perp$,
\[\bar{x}_G := \Proj_{\lin(G)^\perp}(G),\]
and $x_1$, $x_2$ be the coordinates of $\bar{x}_G$ in terms of $v_{F_1, G}$ and $v_{F_2, G}$, 
\[\bar{x}_G = x_1 v_{F_1, G} + x_2 v_{F_2, G}.\]

We can't assume that $v_{F_1,G}$ and $v_{F_2,G}$ form a basis for the lattice $\Lambda_G^*$, however since $v_{F_1,G}$ is a $\Lambda_G^*$-primitive vector, we can set $v_1 := v_{F_1,G}$ and find $v_2 \in \Lambda_G^*$ such that $\{v_1, v_2\}$ is a basis for the lattice $\Lambda_G^*$. Let $h$ and $k$ be the coprime integers such that 
\[v_{F_2, G} = hv_1 + kv_2\]
(they are coprime since $v_{F_2, G}$ is $\Lambda_G^*$-primitive). Substituting $v_2$ by $-v_2$ if necessary, we may assume that $k$ is positive and considering the basis operation $v_2 \mapsto v_2 + av_1$ with $a \in \Z$, we see that we may also choose $v_2$ such that $0 \leq h < k$ (we are essentially using lattice basis reduction, for just dimension $2$).  Adapting an equivalent definition given by Pommersheim~\cite[Section 6]{pommersheim93}, we will say that the cone $\fcone(P,G)$ has \textbf{type} $(h,k)$. 
We defined $h$ and $k$ in terms of the primitive vectors from $\Lambda_G^ *$, however the same values could also have been obtained in terms of a similar relation between the primitive vectors from $\Lambda_G$, see Lemma~\ref{lem:hk}

In order to describe more precisely the building blocks of the quasi-coefficients for both Ehrhart and solid angle polynomials, we consider the usual $r$'th Bernoulli polynomial, defined by the generating function
\begin{equation}\label{eq:Bk-gen}
 \frac{ze^{xz}}{e^z-1} = \sum_{r \geq 0} \frac{B_r(x)}{r!}z^r,
\end{equation}
so that the first couple are given by $B_1(x) = x - 1/2$ and $B_2(x) = x^2 - x + 1/6$.  But here we truncate it, so that it is now supported on the unit interval: $B_r(x) := 0$, for $x \notin [0,1]$.  Now we may define the \textbf{periodized Bernoulli polynomials} as:
\begin{equation}\label{eq:B1-periodized}
 \overline{B}_1(x) := 
  \begin{cases}
    B_1(x -\lfloor x \rfloor) & \quad \text{when } x \notin \Z,\\
    0 & \quad \text{when } x \in \Z,\\
  \end{cases}
\end{equation}
and
\[\overline{B}_r(x) := B_r(x -\lfloor x \rfloor)\]
for all $r > 1$.   

The parameters $h$ and $k$ from $\fcone(P,G)$ play an important role in the following sums. For any $h, k$ coprime positive integers and $x, y \in \R$ the \textbf{Dedekind-Rademacher sum}, introduced by Rademacher~\cite{rademacher64}, is defined as
\begin{equation}\label{eq:dedekind-rademacher}
s(h,k;x,y) := \sum_{r\hspace*{-.2cm} \mod k} \overline{B}_1\lp h \frac{r+y}{k}+x\rp \overline{B}_1\lp\frac{r+y}{k}\rp.
\end{equation}
Note that when $x$ and $y$ are both integers, this sum reduces to the classical \textbf{Dedekind sum} 
\begin{equation*}
s(h,k) := \sum_{r\hspace*{-.2cm} \mod k} \overline{B}_1\lp \frac{rh}{k}\rp \overline{B}_1\lp\frac{r}{k}\rp.
\end{equation*}

\subsection{Statements}

With these local parameters, we obtain the following formula for $a_{d-2}(t)$. We remark that in Theorem \ref{thm:ad2-formula}, each codimension two face $G$ in the summation has its own local geometric data, namely: the type $(h,k)$, and the parameters $x_1, x_2, F_1, F_2 $. 

\begin{restatable}{theorem}{thmadtwoformula}
\label{thm:ad2-formula}
Let $P \subset \R^d$ be a full-dimensional rational polytope. Then for positive real values of $t$, the codimension two quasi-coefficient of the solid angle sum $A_P(t)$ has the following finite form:
\begin{multline*} 
a_{d-2}(t) = \sum_{\substack{G \subset P,\\ \dim G = d-2}} \vol^*(G)
\bigg[
\frac{c_G}{2k}
\bigg(\frac{\|v_{F_2}\|}{\|v_{F_1}\|} \overline{B}_2\big(\innerp{v_{F_1}}{\bar{x}_G} t \big) 
+ \frac{\|v_{F_1}\|}{\|v_{F_2}\|} \overline{B}_2\big(\innerp{v_{F_2}}{\bar{x}_G}t\big)\bigg)\\
+ \lp \omega_P(G) -\frac{1}{4}\rp \bm{1}_{\Lambda_G^*}\lp t\bar{x}_G \rp
- s\big(h,k; (x_1+hx_2)t, -kx_2t\big)
\bigg].
\end{multline*}
\end{restatable}

An important special case of Theorem~\ref{thm:ad2-formula} is the collection of integer polytopes, and the restriction to integer dilations $t$, as follows.

\begin{restatable}{corollary}{corAPtint}
\label{cor:APt-int}
Let $P \subset \R^d$ be a full-dimensional integer polytope. Then for positive integer values of $t$, the codimension two coefficient of the solid angle sum $A_P(t)$ has the following finite form:
\[a_{d-2} = \sum_{\substack{G \subset P,\\ \dim G = d-2}} \vol^*(G)
\bigg[ \frac{c_G}{12k}\lp \frac{\|v_{F_1}\|}{\|v_{F_2}\|} + \frac{\|v_{F_2}\|}{\|v_{F_1}\|} \rp + \omega_P(G) -\frac{1}{4} - s(h,k)\bigg].\]

In particular, for $d = 3$ or $4$, let $P$ be a  full-dimensional integer polytope in $\R^d$. Then for positive integer values of $t$ its solid angle sum is:
\[
 A_P(t) = \vol(P) t^d + \hspace{-.2cm}\sum_{\substack{G \subset P,\\ \dim G = d-2}}\hspace{-.35cm} \vol^*(G)
\bigg[ \frac{c_G}{12k}\lp \frac{\|v_{F_1}\|}{\|v_{F_2}\|} + \frac{\|v_{F_2}\|}{\|v_{F_1}\|} \rp + \omega_P(G) -\frac{1}{4} - s(h,k)\bigg] t^{d-2}.
\]
\end{restatable}

\bigskip
In the last section, we study the question of which rational polytopes $P \subset \R^d$ have the special property that their discrete volumes are equal to their continuous volume. Namely, we would like to know when 
\begin{equation}\label{concrete property}
A_P(t) = \vol (P) t^d,
\end{equation}  
for all integer dilations $t$.  
We exhibit a general family of polytopes that obey such a discrete-continuous property.  In particular, suppose we begin with a rational polytope $P \subset \R^d$, and symmetrize it with respect to the hyperoctahedral group, obtaining an element $Q$ of the polytope group (defined in Section~\ref{sec:concrete}). If $Q$ multi-tiles (see equation \eqref{definition of multitiling}) $\R^d$ by translations, then we prove in Theorem \ref{thm:poly-group} that the original polytope $P$ enjoys property 
\eqref{concrete property}. 
Previously known families of such polytopes arose from tiling (and multi-tiling)~$\R^d$ by translations only.  Here Theorem \ref{thm:poly-group} extends the known families by introducing a non-abelian group. 

Returning to Ehrhart quasi-polynomials, in Section~\ref{sec:Ehrhart} we adapt a technique from Barvinok~\cite{barvinok06} to prove Theorem~\ref{thm:lim-Ap}, showing how the solid angle sum quasi-polynomial of a rational polytope gives the Ehrhart quasi-polynomial, for all positive real $t$. This might seem counter-intuitive at first, because the solid angle sum polynomials are built up from a local metric at each integer point, while the Ehrhart polynomials are purely combinatorial objects. In particular, we obtain the following finite form for the codimension two quasi-coefficient. To state the result, we define the one-sided limits
\begin{equation*}
\overline{B}^+_1(x) := \lim_{\epsilon \to 0^+} \overline{B}_1(x + \epsilon) 
\quad\text{ and }\quad
\overline{B}^-_1(x) := \lim_{\epsilon \to 0^+} \overline{B}_1(x - \epsilon),
\end{equation*}
which differ from $\overline{B}_1(x)$ only at the integers:  $\overline{B}^+_1(n) = - 1/2$ and 
$\overline{B}^-_1(n) = 1/2$ for $n \in \Z$.

\begin{restatable}{theorem}{thmedtwoformula}
\label{thm:ed2-formula}
Let $P \subset \R^d$ be a full-dimensional rational polytope. Then for all positive real values of~$t$, 
the codimension two quasi-coefficient of the Ehrhart function $L_P(t)$ has the following finite form: 
\begin{multline*}
e_{d-2}(t) = \hspace{-.2cm} \sum_{\substack{G \subset P,\\ \dim G = d-2}} \hspace{-.2cm} \vol^*(G) \bigg[\frac{c_G}{2k} \bigg(\frac{\|v_{F_2}\|}{\|v_{F_1}\|} \overline{B}_2\big(\innerp{v_{F_1}}{\bar{x}_G} t \big)
+ \frac{\|v_{F_1}\|}{\|v_{F_2}\|} \overline{B}_2\big(\innerp{v_{F_2}}{\bar{x}_G}t\big)\bigg)\\
-s\big(h,k;(x_1+hx_2)t,-kx_2t\big)
- \frac{1}{2} \bm{1}_{\Z}\lp kx_1t\rp \overline{B}_1\big( (h^{-1}x_1 +x_2)t \big)\\
- \frac{1}{2} \bm{1}_{\Z}(kx_2t)\overline{B}^+_1\big((x_1+hx_2)t\big) 
\bigg],
\end{multline*}
where $h^{-1}$ denotes an integer satisfying $h^{-1}h \equiv 1 \mod k$ if $h \neq 0$ and $h^{-1} := 1$ in case $h = 0$ and $k = 1$.
\end{restatable}

If $P$ is an integer polytope and $t$ is an integer, then 
$\innerp{v_F}{x_F}t \in \Z$, and the formula from Theorem~\ref{thm:ed2-formula} simplifies as follows.

\begin{restatable}{corollary}{coredtwoformula}
\label{cor:ed2-formula}
Let $P \subset \R^d$ be a full-dimensional integer polytope. For positive integer values of $t$, the codimension two coefficient of the Ehrhart polynomial $L_P(t)$ is the following:
\[
e_{d-2} = \sum_{\substack{G \subset P,\\ \dim G = d-2}} \vol^*(G)
\bigg[ \frac{c_G}{12k}\lp \frac{\|v_{F_1}\|}{\|v_{F_2}\|} + \frac{\|v_{F_2}\|}{\|v_{F_1}\|} \rp -s(h,k) + \frac{1}{4}\bigg].
\]
\end{restatable}

We note that it is possible to obtain the latter formulas for the Ehrhart quasi-coefficients using the methods of Berline and Vergne~\cite{berline07} (see \cite[Proposition 31]{berline07} pertaining to a formula corresponding to Theorem~\ref{thm:ed2-formula}, although the correspondence is a nontrivial notational task).

\bigskip
\subsection{Comments about algorithmic aspects}\label{sec:complexity}

In this section we show how to compute the local parameters in the formula from Theorem~\ref{thm:ad2-formula}, provided we are given the hyperplane description of the polytope.  This  formula uses the volumes of the faces of $P$, and we recall that the theoretical complexity of computing volumes of polytopes, from their facet description,  is known to be $\#P$-hard (see~\cite{dyer88}). 
In addition, the formula (Theorem \ref{thm:ad2-formula})  also uses solid-angles, which may be irrational.   We therefore don't make statements about the theoretical complexity of computing with such formula. However we remark that in practice such computations can be approximated (for the solid-angles), especially if the dimension of the polytope is fixed (see~\cite{dyer88}).

For an integer vector $x \in \Z^d$, denote by $\gcd(x)$ the greatest common divisor of its entries. Let the defining inequalities of the two facets $F_1$ and $F_2$ incident to a $(d-2)$-dimensional face $G$ be $\innerp{a_1}{x} \leq b_1$ and $\innerp{a_2}{x} \leq b_2$, with $a_1, a_2 \in \Z^d$ and $b_1, b_2 \in \Z$. Since $a_j$ is an outward-pointing normal vector to $F_j$, we can compute $v_{F_j} = \frac{1}{\gcd(a_j)}a_j $. Hence $c_G :=  -\innerp{N_P(F_1)}{N_P(F_2)} = -\frac{\innerp{v_{F_1}}{v_{F_2}}}{\|v_{F_1}\|\|v_{F_2}\|}$.

Next we show how a lattice basis for $\Lambda_G^*$ can be computed. We observe that by Lemma~\ref{lm:dual-lattice} below, $\Lambda_G^*$ corresponds to the orthogonal projection of $\Z^d$ onto $\lin(G)^\perp$.
Denoting the $d \times 2$ matrix with columns $v_{F_1}$ and $v_{F_2}$ by $U$, we have that $P = U(U^\trans U)^{-1}U^\trans$ is the orthogonal projection onto $\lin(G)^{\perp}$. Indeed, one can check directly that $PU = U$, $P^2 = P$ and $Pv = 0$ for any $v \in \lin(G)$. Therefore the columns of $P$ generate $\Lambda_G^*$. From a set of generating vectors, one can compute a lattice basis by an application of the LLL-algorithm (as described by Buchmann and Pohst~\cite{buchmann89}).

Let $m = 1$ and $j = 2$, or $m = 2$ and $j = 1$.  We now proceed to compute $v_{F_m, G}$, the $\Lambda_G^*$-primitive vector in the direction of $N_{F_m}(G)$. Let 
\begin{equation}\label{eq:fjk}
f_{m, j} := \innerp{v_{F_m}}{v_{F_m}}v_{F_j} - \innerp{v_{F_m}}{v_{F_j}}v_{F_m}.
\end{equation}
It is an integer vector in $\lin(G)^\perp$ orthogonal to $v_{F_m}$ and since $\innerp{f_{m,j}}{v_{F_j}} > 0$ (by Cauchy-Schwarz), it is a vector in the same direction of $N_{F_m}(G)$. Since $f_{m,j} \in \Z^d \cap \lin(G)^\perp \subseteq \Lambda_G^*$, it has integral coordinates in the computed basis for $\Lambda_G^*$. Computing them and dividing by their $\gcd$, we get $v_{F_m, G}$.

Having a lattice basis for $\Lambda_G^*$, its determinant $\det(\Lambda_G^*)$ can be computed directly. Also, using $v_{F_1, G}$ and $v_{F_2, G}$, we can compute $v_2$ such that $v_1 := v_{F_1, G}$ and $v_2$ is a lattice basis. Hence we can also compute $h$ and $k$. To compute $x_1$ and $x_2$, we can use $\bar{x}_G = P x_G$ if we already know a point $x_G \in G$ and then write $\bar{x}_G$ in terms of $v_{F_1, G}$ and $v_{F_2, G}$. More generally, we observe that for any point $x_G \in G$ we must have $\innerp{a_1}{x_G} = b_1$ and $\innerp{a_2}{x_G} = b_2$, so:
\begin{align*}
b_2 &= \innerp{a_2}{\bar{x}_G} = \|a_2\| \innerp{N_P(F_2)}{x_1 v_{F_1, G} + x_2 v_{F_2, G}}\\
&= x_1 \|a_2\| \|v_{F_1, G}\| \innerp{N_P(F_2)}{N_{F_1}(G)} 
= x_1 \|a_2\| \|v_{F_1, G}\| \sqrt{1-c_G^2}\\ 
&= x_1 \|a_2\| |\det(v_{F_1,G},v_{F_2,G})| / \|v_{F_2, G}\| 
= x_1 \|a_2\| k  / \|v_{F_2}\|,
\end{align*}
(see the proof of Lemma~\ref{lem:hk}) thus
\[x_1 = \frac{b_2}{ k \gcd(a_2)},\quad \text{and analogously,}\quad x_2 = \frac{b_1}{ k \gcd(a_1)}.\]

The Dedekind-Rademacher sums can be computed efficiently by proceeding as in the Euclidean algorithm, see Rademacher~\cite{rademacher64}.

\bigskip
\section{Preliminaries}\label{sec:pre}

The current section contains some definitions and background on known results, which will be useful in proving our main results.  

\subsection{Lattices}\label{sec:lattices}

A {\bf $k$-dimensional lattice} $L$ in $\R^d$ is a discrete additive subgroup generated by any $k$ linearly independent vectors in $\R^d$.  Any set of $k$ vectors that generates $L$ is called a {\bf lattice basis}. 
The {\bf determinant} $\det(L)$ of $L$ is the $k$-dimensional volume of any fundamental domain for $L$.  It is easy to compute the volume of $L$ :  If $B \in \R^{d \times k}$ is a matrix whose columns are formed by a lattice basis of $L$,  then it is a standard fact that
\begin{equation} \label{det.sublattice}
\det(L) = \det(B^{\sf T}B)^{1/2}. 
\end{equation}
Due to this relation, we also use the notation $ |\det(B)| := \det(B^{\sf T}B)^{1/2}$.

The {\bf dual lattice} $L^*$ is defined as
\[ L^* := \{y \in \spann(L) : \innerp{x}{y}\in \Z \text{ for all } x \in L\}.\]
We will always use {\bf span($L$)} to mean that we are taking the span over $\mathbb R$, 
so that $\spann(L)$ is always a vector space over $\mathbb R$. Next 
we assume that $L$ is a subset of another lattice $\Lambda \subset \R^d$ and define
\[L^\perp := \{v \in \Lambda^* : \innerp{v}{x} = 0 \text{ for all } x \in L\},\]
note that  $L^\perp \subseteq \Lambda^*$. The lattice $L$ is called a {\bf primitive lattice}, with respect to $\Lambda$, when 
\[\spann(L) \cap \Lambda = L.\]
Next we state two well known facts about lower-dimensional lattices (i.e. when they do not have full rank) and defer the proofs to an appendix.  The first lemma is always used when translating absolute volumes to relative volumes.

\begin{restatable}{lemma}{detLperp}
\label{lm:det-Lperp}
Let $\Lambda \subset \R^d$ be a $d$-dimensional lattice and let $L \subseteq \Lambda$ be a primitive lattice with respect to $\Lambda$. Then 
\[\det(L^\perp) = \frac{ \det(L)}{ \det(\Lambda)}.\]
\end{restatable}

We note that Lemma~\ref{lm:det-Lperp} is non-trivial even in the case that $\Lambda := \Z^{d}$ and $L$ is a $(d-1)$-dimensional sublattice. 

\begin{restatable}{lemma}{duallattice}
\label{lm:dual-lattice}
Let $\Lambda \subset \R^d$ be a $d$-dimensional lattice and $L \subseteq \Lambda$ be a primitive lattice with respect to $\Lambda$. Then \[L^* = \Proj_{\spann(L)}(\Lambda^*).\]
\end{restatable}

\medskip
\subsection{Fourier analysis}

Let $\hat{f}$ and $\Fcal(f)$ denote the \textbf{Fourier transform} of a function $f \colon \R^d \to \C$, which is defined as
\[\hat{f}(\xi) := \Fcal(f)(\xi) := \int_{\R^d} f(u) e^{-2\pi i \innerp{u}{\xi}} 
\diff u.\]
For $x \in \R^d$ let the \textbf{translation by $x$} be $T_x(\xi) := \xi - x$, and the \textbf{convolution} between two functions be $(f \ast g)(x) := \int_{\R^d} f(y)g(x-y) \diff y$.
We recall the following standard identities for the Fourier transform (see e.g.,~\cite[Chapter~I, Theorem 1.4]{stein71}), where $M \in \R^{d \times d}$ is an invertible matrix:
\begin{align}
\Fcal(f \circ T_x)(\xi) &= \hat{f}(\xi)e^{-2\pi i \innerp{x}{\xi}},\label{eq:fourier1}\\
\Fcal(f \ast g)(\xi) &= \hat{f}(\xi) \hat{g}(\xi),\label{eq:fourier2}\\
(\hat{f} \circ M^\trans)(\xi) &= \frac{1}{|\det(M)|}\Fcal(f \circ M^{-1})(\xi).\label{eq:fourier3}
\end{align}

The following theorem, known as the Poisson summation formula, is one of the main tools in Fourier analysis (Stein and Weiss~\cite[Chapter~VII, Corollary~2.6]{stein71}).

\begin{theorem}[Poisson summation]\label{thm:poisson}
Let $f \colon \R^d \to \C$ be a function that enjoys the following two decay conditions and $\Lambda \subset \R^d$ be a $d$-dimensional lattice. Suppose there exist positive constants $\delta$, $C$ such that for all $x \in \R^n$:
 
 (a) \  $|f(x)| < C(1+|x|)^{-n-\delta}$ 
 
 (b) \ $|\hat{f}(x)| < C(1+|x|)^{-n-\delta}$.

Then for any $y \in \R^d$,
\begin{equation}  \label{PoissonSummation}
\det(\Lambda)\sum_{x \in \Lambda} f(x+y) =  \sum_{\xi \in \Lambda^*}\hat{f}(\xi)e^{2\pi i \innerp{y}{\xi}},
\end{equation}
and both sides of  \eqref{PoissonSummation}  converge absolutely.
\end{theorem}

A function that will play a special role in this work is the \textbf{Gaussian function}, that for $\epsilon > 0$ is 
\[
\phi_{d, \epsilon}(x) := \epsilon^{-d/2}e^{-\pi\|x\|^2/\epsilon}.
\]
Its Fourier transform is (see e.g.,~\cite[Chapter I, Theorem 1.13]{stein71}) 
\[
\hat{\phi}_{d,\epsilon}(\xi) = e^{-\epsilon\pi\|\xi\|^2}.
\]
Note that $\hat{\phi}_{d,\epsilon}$ doesn't explicitly depend on $d$ (except for the $2$-norm in $\mathbb R^d$) so we also denote it by $\hat{\phi}_\epsilon$.

\subsection{Fourier transforms of polytopes and solid angle sums}\label{sec:summary-diaz16}
In this section we present a summary of the main results from Diaz, Le, and Robins~\cite{diaz16}. 

A solid angle at any point $x \in P$, which is also the volume of a local spherical polytope, has an analytical representation that is convenient for our purposes.   To  introduce it, let $\bm{1}_P$ denote the indicator function of $P$, that is $\bm{1}_P(x) = 1$ if $x \in P$ and $\bm{1}_P(x) = 0$ if $x \notin P$. The solid angle can be computed as the limit of the convolution between the Gaussian $\phi_{d, \epsilon}$ and the indicator function of $P$ (cf. Diaz, Le and Robins~\cite[Lemma 1]{diaz16}):
\begin{equation}\label{eq:conv-solid-angle}
\omega_P(x) = \lim_{\epsilon \to 0^+} \int_P \phi_{d,\epsilon}(t-x) \diff t = \lim_{\epsilon \to 0^+}(\bm{1}_P \ast \phi_{d,\epsilon})(x).
\end{equation}
More generally, if we replace $\bm{1}_P$ by a continuous function supported on $P$, we have the following (cf. Diaz, Le and Robins~\cite[Lemma~3 and Theorem~5]{diaz16}):

\begin{lemma}\label{thm:conv-series}
Let $P$ be a full-dimensional polytope in $\R^d$ and $f$ be a continuous function on $P$ and zero outside $P$. Then for all $x \in \R^d$,
\[ \lim_{\epsilon \to 0^+}(f \ast \phi_{d,\epsilon})(x) = f(x) \omega_P(x).\]
Moreover,
\[\sum_{x \in \Z^d} f(x) \omega_P(x) = \lim_{\epsilon \to 0^+}\sum_{x \in \Z^d}(f \ast \phi_{d,\epsilon})(x).\]
\end{lemma}
Note that the left-hand side of the above identity is a finite sum since $P$ is compact and $\omega_P(x) = 0$ for $x \notin P$, while the right-hand side is the limit of an infinite series.

The method from Diaz, Le, and Robins consists of two steps: First, the solid angles are written with convolutions and the solid angle sum is represented with the series from Lemma~\ref{thm:conv-series}, next the Poisson summation formula is applied to represent $A_P(t)$ as a series with the Fourier transform of~$P$, leading to (cf. Diaz, Le, and Robins~\cite[Lemma 2]{diaz16}):

\begin{lemma}\label{lm:APt-hat1P}
Let $P$ be a full-dimensional polytope $P$ in $\R^d$ and $t$ any positive real number. Then the solid angle sum of $P$ can be written as follows:
\[A_P(t) = t^d \lim_{\epsilon \to 0^+} \sum_{\xi \in \Z^d} \hat{\bm{1}}_P(t\xi)e^{-\pi\epsilon\|\xi\|^2}.\]
\end{lemma}

Through successive applications of Stokes formula (in the frequency space of Poisson summation), the Fourier transform of $P$ is then written as a sum over the faces of $P$~\cite[Theorem 1]{diaz16}. By treating these terms carefully, keeping track of 'generic' and 'nongeneric' frequency vectors on the right-hand-side of Poissson summation, one can find local formulas for the coefficients $a_{d-k}(t)$.

If $F$ is a face of $P$, let $\Proj_F$ be the orthogonal projection onto $\lin(F)$. If $G$ is a facet of~$F$, denote by $N_F(G)$ the unit normal vector in $\lin(F)$ pointing outward to~$G$ and define the weight on the pair $(F,G)$:
\[W_{(F,G)}(\xi) := \frac{-1}{2\pi i}\frac{\innerp{\Proj_F(\xi)}{N_F(G)}}{\| \Proj_F(\xi)\|^2}.\]

The \textbf{face poset} $G_P$ of $P$ consists of all faces of $P$ ordered by inclusion and a \textbf{chain} $T$ of \textbf{length} $l(T) = k$ is a sequence of faces $T = (F_0 \to F_1 \to F_2 \to \dots \to F_k)$ with $F_0 = P$ and $F_j$ a facet of $F_{j-1}$ for every $j$. 

The \textbf{admissible set} $S(T)$ of a chain is the set of all vectors orthogonal to $\lin(F_k)$ but not to $\lin(F_{k-1})$. For a point $\xi \in S(T)$, the \textbf{rational weight} $\Rcal_T(\xi)$ is the product
\begin{equation}\label{eq:rational-weight}
\Rcal_T(\xi) := \vol(F_k) \prod_{j=1}^k W_{(F_{j-1}, F_j)}(\xi),
\end{equation}
where the volume of $F_k$ is the $(d-k)$-dimensional volume and the \textbf{exponential weight}~is
\begin{equation}\label{eq:exponential-weight}
\Ecal_T(\xi) := e^{-2\pi i\innerp{\xi}{x_{F_k}}},
\end{equation}
where $x_{F_k}$ is any point from $F_k$, the last face from chain $T$. Note that since $\xi \in S(T)$, the value of $\innerp{\xi}{x_{F_k}}$ does not depend on the choice of $x_{F_k}$.

The next result from Diaz, Le, and Robins gives a formula for~$a_k(t)$, for any positive real $t$:
\begin{theorem}   \cite[Theorem 2]{diaz16}     \label{thm:ak-formula}
Let $P$ be a full-dimensional rational polytope in $\R^d$, and $t$ be a positive real number. Then we have $A_P(t) = \sum_{k = 0}^d a_k(t)t^k$, where, for $0 \leq k \leq d$,
\begin{equation*}
a_k(t) = \lim_{\epsilon \to 0^+} \sum_{T :\, l(T) = d-k}\, \sum_{\xi \in \Z^d \cap S(T)} \Rcal_T(\xi) \Ecal_T(t\xi) \hat{\phi}_\epsilon(\xi).
\end{equation*}
\end{theorem}

Using this theorem one can get more explicit formulas for the coefficients, although their complexity increases with the length of the chains considered. 
For the quasi-coefficient $a_{d-1}(t)$, we have the following known formula, given in terms of the facets of $P$ and the periodized Bernoulli polynomial $\overline{B}_1$:

\begin{theorem}   \cite[Theorem 3]{diaz16}  \label{thm:ad1-formula}
Let $P$ be a full-dimensional rational polytope. Then the codimension one quasi-coefficient of the solid angle sum $A_P(t)$ has the following local formula for all positive real values of $t$:
\[ a_{d-1}(t) = -\hspace{-.3cm}\sum_{\substack{F \subset P,\\ \dim(F) = d-1}}\hspace{-.3cm} \vol^*(F)\overline{B}_1\big(\innerp{v_F}{x_F}t\big),\]
where $x_F$ is any point in $F$ and $v_F$ is the primitive integer vector in the direction of $N_P(F)$.
\end{theorem}

\bigskip
\section{Lattice sums}\label{sec:gl-sums}

Let $\Lambda$ be a $k$-dimensional lattice in $\R^d$, $w_1, \dots, w_k$ be linearly independent vectors from $\Lambda^*$ and $W \in \R^{d \times k}$ be a matrix with them as columns. For a $k$-uple $e = (e_1, \dots, e_k)$ of positive integers, let $|e| := \sum_{j = 1}^k e_j$. For all $x \in \R^d$, our goal in this section is to evaluate
\begin{equation} \label{eq:main-limit}
L_\Lambda(W, e; x) :=
\lim_{\epsilon \rightarrow 0^+} 
\frac{1}{(2\pi i)^{|e|}} \sum_{\substack{\xi \in \Lambda:\\ \innerp{w_j}{\xi} \neq 0, \forall j}} \frac{e^{-2\pi i \innerp{x}{\xi}}}{\prod_{j = 1}^k \innerp{w_j}{\xi}^{e_j}} e^{-\pi \epsilon \|\xi\|^2}.
\end{equation}
These limits of lattice sums come up in the development of the expression in Theorem~\ref{thm:ak-formula}, they also appear in the work of Witten, on $2$-dimensional gauge theory~\cite[pp.~363]{witten92}, and a similar expression with a different limit process instead of the Gaussian factor is called a \emph{Dedekind sum} by Gunnels and Sczech~\cite{gunnels03}. This name is justified since the expression obtained in Theorem~\ref{cor:lattice-sum2} can be written as a Dedekind-Rademacher sum in the case that $e = (1,1)$ (cf. Section~\ref{sec:interior-points}).

For any $k$-uple $e = (e_1, \dots, e_k)$ of positive integers we define a \textbf{$k$-dimensional Bernoulli polynomial} $\Bcal_e \colon \R^k \to \R$ as
\[\Bcal_e(x) := B_{e_1}(x_1) \cdots B_{e_k}(x_k).\]
Note that $\Bcal_e$ is supported in $[0,1]^k$. The reason for defining these polynomials is that their Fourier transforms, evaluated at integer inputs, are the inverse of products of linear forms, as stated in the lemma below (see e.g. Apostol~\cite[Theorem 12.19]{apostol76}).

\begin{lemma}\label{lem:fourier-bernoulli}
For all $r \geq 1$, the Fourier transform of the Bernoulli polynomial $B_r(x)$ satisfies:
\[
\hat{B_r}(n) = 
\begin{cases}
0,\quad \text{if } n = 0,\\
-\frac{r!}{(2\pi i)^r n^r}\quad \text{if } n \in \Z \setminus \{0\}.
\end{cases}
\]
Thus for any $k$-uple $e = (e_1, \dots, e_k)$ of positive integers,
\[
\hat{\Bcal_e}(m) =
\begin{cases}
0,\quad \text{if } m_j = 0 \text{ for some }j,\\
\frac{(-1)^de_1!\cdots e_k!}{(2\pi i)^{|e|} m_1^{e_1}\cdots m_k^{e_k} }\quad 
\text{if } m \in (\Z \setminus \{0\})^k.
\end{cases}
\]
\end{lemma}

Returning to the evaluation of $L_\Lambda(W,e;x)$, we assume first that $\Lambda$ is the full-dimensional integer lattice $\Z^d$; in this case $W$ is invertible. Let $P_{W,x}$ be the parallelepiped
\[P_{W,x} := \{n \in \R^d : W^{-1}(n-x) \in [0,1]^d\} = x + W[0,1]^d.\]
We prove the following theorem, which gives a finite form for \eqref{eq:main-limit}, in terms of a sum over the integer points in $P_{W,x}$ and the $d$-dimensional Bernoulli polynomial times a local solid angle.

\begin{theorem}\label{thm:lattice-sum}
If $W \in \Z^{d \times d}$ is an invertible matrix with columns $w_1, 
\dots, w_d$, $e = (e_1, \dots, e_d)$ is a $d$-uple of positive integers and $x \in 
\R^d$, then:
\[L_{\Z^d}(W,e;x) = \frac{(-1)^d}{e_1!\cdots e_d! |\det(W)|} 
\sum_{n \in \Z^d \cap P_{W,x}} \Bcal_e\big(W^{-1}(n-x)\big)\omega_{P_{W,x}}(n).\]
\end{theorem}

\begin{proof}
We recognize each term inside sum~\eqref{eq:main-limit} as the Fourier transform of a function, apply Poisson summation and then use Lemma~\ref{thm:conv-series} to compute the limit.

Using Lemma~\ref{lem:fourier-bernoulli} and identity~\eqref{eq:fourier3} with $\Bcal_e$ and $W$, for any $\xi \in \Z^d$ such that $\innerp{w_j}{\xi} \neq 0$ for all $j$, we have:
\[ \frac{1}{(2\pi i)^{|e|}\prod_{j = 1}^d \innerp{w_j}{\xi}^{e_j}} =
\frac{(-1)^d}{e_1!\cdots e_d!}(\hat{\Bcal_e}\circ W^\trans)(\xi) =
\frac{(-1)^d}{e_1!\cdots e_d!|\det(W)|}\Fcal(\Bcal_e \circ W^{-1})(\xi).\]

To obtain the same term that appears in~\eqref{eq:main-limit}, we make use of identity~\eqref{eq:fourier1} and recall that $\hat{\phi}_\epsilon(\xi) = e^{-\pi \epsilon \|\xi\|^2}$. Further noticing that $(\hat{\Bcal_e}\circ W^\trans)(\xi) = 0$ when $\innerp{w_j}{\xi} = 0$ for some $j$, we have:
\[L_{\Z^d}(W,e;x) =  \lim_{\epsilon \rightarrow 0^+} \frac{(-1)^d}{e_1!\cdots e_d! |\det(W)|} \sum_{\xi \in \Z^d} \Fcal(\Bcal_e \circ W^{-1} \circ 
T_x)(\xi)\hat{\phi_\epsilon}(\xi).\]
Using identity~\eqref{eq:fourier2} and Poisson summation (Theorem~\ref{thm:poisson}),
\[L_{\Z^d}(W,e;x) = \lim_{\epsilon \rightarrow 0^+} \frac{(-1)^d}{e_1!\cdots e_d! |\det(W)|} \sum_{n \in \Z^d} \big((\Bcal_e \circ W^{-1} \circ T_x)\ast 
\phi_{d,\epsilon}\big)(n).\]
Note that the support of $\Bcal_e \circ W^{-1} \circ T_x$ is exactly $P_{W, x}$ and $(\Bcal_e \circ W^{-1} \circ T_x)(n) = \Bcal_e\big(W^{-1}(n-x)\big)$. This enables us to use Lemma~\ref{thm:conv-series} and obtain

\[L_{\Z^d}(W,e;x) = \frac{(-1)^d}{e_1!\cdots e_d! |\det(W)|} \sum_{n \in \Z^d \cap P_{W,x}} 
\Bcal_e\big(W^{-1}(n-x)\big)\omega_{P_{W,x}}(n). \qedhere\]
\end{proof}

The situation is almost the same for the general case where $\Lambda$ is a $k$-dimensional lattice in $\R^d$, however in this case we must restrict attention to the subspace spanned by~$\Lambda$. Note that $W \in \R^{d \times k}$ is not invertible but when we see it as a linear transformation $W \colon \R^k \to \spann(\Lambda)$ it is, such inverse is called the \textbf{pseudoinverse} and can be computed as $W^{+} = (W^\trans W)^{-1}W^\trans$. Furthermore, it follows that $WW^{+}$ is the orthogonal projection $\Proj_{\spann(\Lambda)}$ from $\R^d$ to $\spann(\Lambda)$. The parallelepiped $P_{W,x}$ becomes a $k$-dimensional parallelepiped in $\spann(\Lambda)$:
\begin{equation*}
P_{W,x} := \{n \in \spann(\Lambda) : W^{+}(n-x) \in [0,1]^k\} = \Proj_{\spann(\Lambda)}(x) + W[0,1]^k.
\end{equation*}
Identity~\eqref{eq:fourier3} also has to be adapted, since we are dealing with a $k$-dimensional subspace embedded in $\R^d$. More specifically, for $f \colon \R^k \to \C$ and $\xi \in \spann(\Lambda)$, in place of~\eqref{eq:fourier3} we use:
\begin{equation*}
\det(W^\trans W)^{1/2}\int_{\R^k} f(y)e^{-2\pi i \innerp{y}{W^\trans \xi}}\diff y
= \int_{\spann(\Lambda)}f(W^+x)e^{-2\pi i \innerp{x}{\xi}}\diff x.
\end{equation*}
With these remarks, the same proof of the previous theorem gives:

\begin{theorem}\label{cor:lattice-sum2}
If $W \in \R^{d \times k}$ is a matrix with linearly independent columns $w_1, \dots$, $w_k \in \Lambda^*$, $e = (e_1, \dots, e_k)$ is a $k$-uple of positive integers and $x \in \R^d$, then:
\[L_{\Lambda}(W,e;x) = \frac{(-1)^k}{e_1!\cdots e_k! \det(W^\trans W)^{1/2} \det(\Lambda)} \sum_{n \in \Lambda^* \cap P_{W,x}} \Bcal_e\big(W^{+}(n-x)\big)\omega_{P_{W,x}}(n).\]
\end{theorem}

\begin{remark}
Theorem~\ref{thm:lattice-sum} is similar to Proposition 2.7 of Gunnels and Sczech~\cite{gunnels03}. The main difference between these two results is that the theorem above uses solid angle weights but when all $e_j > 1$, the sum in~\eqref{eq:main-limit} is absolutely convergent for $\epsilon = 0$ and we may interchange the limit with the lattice sum. The resulting sum is then equal to the Dedekind sum considered by Gunnels and Sczech, and Theorem~\ref{thm:lattice-sum} can be compared with their Proposition 2.7.
\end{remark}

\bigskip
\section{Proofs of Theorem~\ref{thm:ad2-formula} and Corollary  \ref{cor:APt-int}}   \label{sec:ad2}

We start with a lemma that shows how the 'type $(h,k)$' simultaneously describes the relation of $\fcone(P,G)$ with respect to $\Lambda_G$ and with $\Lambda_G^*$.

\begin{lemma}\label{lem:hk}
If $h$ and $k$ are such that $v_1 := v_{F_1,G}$ and $v_2 := (v_{F_2,G} - h v_{F_1,G})/k$ form a lattice basis for $\Lambda_G^*$ (as defined above), then 
\[
u_1 := v_{F_1}   \text{ and  }  u_2 := (v_{F_2} + h v_{F_1})/k
\]
 form a lattice basis for $\Lambda_G$. In particular, 
\[
k = 
\frac{|\det(v_{F_1,G}, v_{F_2,G})|}{\det(\Lambda_G^*)} 
= \frac{|\det(v_{F_1}, v_{F_2})|}{ \det(\Lambda_G)}.
\]
\end{lemma}

\begin{proof}
We have to prove that $\innerp{v_{F_1,G}}{v_{F_2}} = \innerp{v_{F_2,G}}{v_{F_1}} = k$. Using this, the lemma follows directly from the following computation:
\begin{multline*}
(v_2, v_1)^\trans (u_1, u_2) =  \lp\begin{smallmatrix} -h/k & 1/k\\ 1 & 0 \end{smallmatrix}\rp (v_{F_1,G}, v_{F_2,G})^\trans (v_{F_1}, v_{F_2}) \lp\begin{smallmatrix} 1 & h/k\\ 0 & 1/k \end{smallmatrix}\rp\\ 
=  \lp\begin{smallmatrix} -h/k & 1/k\\ 1 & 0 \end{smallmatrix}\rp \lp\begin{smallmatrix} 0 & k\\ k & 0 \end{smallmatrix}\rp \lp\begin{smallmatrix} 1 & h/k\\ 0 & 1/k \end{smallmatrix}\rp
= k \lp\begin{smallmatrix} -h/k & 1/k\\ 1 & 0 \end{smallmatrix}\rp  \lp\begin{smallmatrix} 0 & 1/k\\ 1 & h/k \end{smallmatrix}\rp = \lp\begin{smallmatrix} 1 & 0\\ 0 & 1 \end{smallmatrix}\rp.
\end{multline*}

Since we work simultaneously with two orthonormal basis $\{N_P(F_1), N_{F_1}(G)\}$ and $\{N_P(F_2), N_{F_2}(G)\}$ for $\lin(G)^\perp$, it is useful to know how they are related. From the outward orientation of the normal vectors (see Figure~\ref{fig:rel-orient}), we have
\begin{align}\label{eq:unit-vectors}
\begin{aligned}
N_P(F_2) &= -c_G N_P(F_1) + \sqrt{1-c_G^2}N_{F_1}(G), \text{ and}\\ 
N_{F_2}(G) &= \sqrt{1-c_G^2}N_P(F_1) + c_GN_{F_1}(G).
\end{aligned}
\end{align}

\begin{figure}[t]
\begin{tikzpicture}[scale = .8]
\draw[thick] (2.2,3.3) -- (1.368,2.084) node[below] {\small $F_1$};
\draw[thick] (2.2,3.3) -- (2.008,1.828) node[below] {\small $F_2$};
\draw[->, thick, blue] (2.2,3.3) -- (2.902,4.326) node[right] {\small $N_{F_1}(G)$};
\draw[->, thick, blue] (2.2,3.3) -- (1.174,4.002) node[above] {\small $N_{P}(F_1)$};
\draw[->, thick, red] (2.2,3.3) -- (2.362,4.542) node[above] {\small $N_{F_2}(G)$};
\draw[->, thick, red] (2.2,3.3) -- (3.442,3.138) node[above] {\small $N_{P}(F_2)$};

\draw[thick] (8.4,2.2) -- (6.4,2) node[below] {\small $F_1$};
\draw[thick] (8.4,2.2) -- (10.2,1.8) node[below] {\small $F_2$};
\draw[->, thick, blue] (8.4,2.2) -- (10,2.36) node[above] {\small $N_{F_1}(G)$};
\draw[->, thick, blue] (8.4,2.2) -- (8.24,3.8) node[left] {\small $N_{P}(F_1)$};
\draw[->, thick, red] (8.4,2.2) -- (6.96,2.52) node[above] {\small $N_{F_2}(G)$};
\draw[->, thick, red] (8.4,2.2) -- (8.72,3.64) node[right] {\small $N_{P}(F_2)$};
\end{tikzpicture}
\caption{Relative orientations between the normal vectors of each facet.}\label{fig:rel-orient}
\end{figure}

We prove $\innerp{v_{F_1,G}}{v_{F_2}} = k$, since the proof for the other inner product is the same. The $\Lambda_G$-primitive vector $v_{F_2}$ along $N_P(F_2)$ and the $\Lambda_G^*$-primitive vector $v_{F_2,G}$ along $N_{F_2}(G)$ have a special relation. Using Lemma~\ref{lm:det-Lperp} with $\Lambda := \Lambda_G$ and $L$ as the one dimensional lattice spanned by $v_{F_2}$, we get
\begin{equation}\label{eq:vF2-F2G}
\|v_{F_2}\| = \det(\Lambda_G) \|v_{F_2,G}\|.
\end{equation}
Next we establish an identity developing $\det(v_{F_1,G}, v_{F_2,G})$ in two ways:
\begin{multline*}
\det(v_{F_1,G}, v_{F_2,G})^2 = \det\big( (v_{F_1,G}, v_{F_2,G})^\trans(v_{F_1,G}, v_{F_2,G})\big)\\ = \|v_{F_1,G}\|^2\|v_{F_2,G}\|^2 - \innerp{v_{F_1,G}}{v_{F_2,G}}^2
= \|v_{F_1,G}\|^2\|v_{F_2,G}\|^2(1 - c_G^2),
\end{multline*}
and
\begin{multline*}
\det(v_{F_1,G}, v_{F_2,G})^2 = \det\big( (v_{F_1,G}, v_{F_2,G})^\trans(v_{F_1,G}, v_{F_2,G})\big)\\ = \det\Big( \lp\begin{smallmatrix} 1 & h\\ 0 & k \end{smallmatrix}\rp^\trans (v_{1}, v_{2})^\trans (v_{1}, v_{2}) \lp\begin{smallmatrix} 1 & h\\ 0 & k \end{smallmatrix}\rp\Big) = k^2 \det(\Lambda_G^*)^2 = k^2/ \det(\Lambda_G)^2.
\end{multline*}
Finally,
\begin{multline*}
\innerp{v_{F_1,G}}{v_{F_2}} = \|v_{F_1,G}\|\|v_{F_2}\|\innerp{N_{F_1}(G)}{N_P(F_2)}\\ = \|v_{F_1,G}\|\|v_{F_2,G}\|\det(\Lambda_G)\sqrt{1-c_G^2} = k.\qedhere
\end{multline*}
\end{proof}

We now proceed to the main result of the paper, whose proof is somewhat longer, and is subdivided into several sections.
\thmadtwoformula* 

\medskip
\subsection{Proof of Theorem~\ref{thm:ad2-formula}}
We start with the formula from Theorem~\ref{thm:ak-formula} and consider all chains $(P \to F \to G)$ of length $2$:
\begin{multline*}
 a_{d-2}(t) = \lim_{\epsilon \to 0^+}\frac{1}{(-2\pi i)^2} \sum_{F \subset P} \sum_{G \subset F} \vol(G)\\
\sum_{\xi \in \Lambda_G \setminus \Lambda_F} \frac{\innerp{\xi}{N_P(F)} \innerp{\Proj_F(\xi)}{N_F(G)}}{\innerp{\xi}{\xi} \innerp{\Proj_F(\xi)}{\Proj_F(\xi)}} 
e^{-2\pi i \innerp{\xi}{t \bar{x}_G}} \hat{\phi}_\epsilon(\xi).
\end{multline*}

Since $N_P(F)$ and $N_F(G)$ form an orthonormal basis for $\lin(G)^{\perp}$, for $\xi \in \lin(G)^\perp$, we have $\Proj_F(\xi) = \innerp{\xi}{N_F(G)}N_F(G)$ and we can simplify the expression above with
\[
\frac{\innerp{\xi}{N_P(F)} \innerp{\Proj_F(\xi)}{N_F(G)}}{\innerp{\xi}{\xi} \innerp{\Proj_F(\xi)}{\Proj_F(\xi)}} 
= \frac{\innerp{\xi}{N_P(F)}}{\innerp{\xi}{N_F(G)}\innerp{\xi}{\xi}}.
\]

Denoting by $F_1$ and $F_2$ the two facets incident to a face $G$ of dimension $d-2$, we switch the order of the sums to obtain
\begin{multline*}
 a_{d-2}(t) = \sum_{\substack{G \subset P\\ \dim(G) = d-2}} \lim_{\epsilon \to 0^+}\frac{\vol(G)}{(-2\pi i)^2} \sum_{j = 1}^2 
\sum_{\xi \in \Lambda_G \setminus \Lambda_{F_j}} \frac{\innerp{\xi}{N_P(F_{j})} e^{-2\pi i \innerp{\xi}{t \bar{x}_G}}}{\innerp{\xi}{N_{F_j}(G)}\innerp{\xi}{\xi}} \hat{\phi}_\epsilon(\xi).
\end{multline*}

It follows from Lemma~\ref{lm:det-Lperp} that $\det(\lin(G) \cap \Z^d) = \det(\lin(G)^{\perp} \cap \Z^d) =: \det(\Lambda_G)$. We note that we are using here the property that $P$ is a rational $d$-dimensional polytope, so that $(\lin(G) \cap \Z^d)^\perp = \lin(G)^{\perp} \cap \Z^d$.  We therefore conclude that  
\[\vol^*(G) :=  \frac{\vol(G)}{\det(\lin(G) \cap \Z^d)} =  \frac{\vol(G)}{\det(\Lambda_G)}. \] 
We decompose the expression into three distinct sums:
\begin{equation}\label{eq:ad2t}
a_{d-2}(t) = \sum_{\substack{G \subset P\\ \dim(G) = d-2}} \vol^*(G)\big( b_1(G;t) + b_2(G;t) + c(G;t)\big),
\end{equation}
where for $j = 1$ and $m = 2$, or for $j = 2$ and $m = 1$, we define
\[
b_j(G;t) := \lim_{\epsilon \to 0^+}\frac{\det(\Lambda_G)}{(-2\pi i)^2} \sum_{\xi \in \Lambda_{F_m}\setminus \{0\}} \frac{\innerp{\xi}{N_P(F_{j})} e^{-2\pi i \innerp{\xi}{t \bar{x}_G}}}{\innerp{\xi}{N_{F_j}(G)}\innerp{\xi}{\xi}} \hat{\phi}_\epsilon(\xi),
\]
and
\begin{equation*}
 c(G;t) := \lim_{\epsilon \to 0^+}\frac{\det(\Lambda_G)}{(-2\pi i)^2} \hspace*{-.4cm}
 \sum_{\substack{\vspace{.1cm}\\\xi \in \Lambda_{G}\setminus (\Lambda_{F_1} \cup \Lambda_{F_2})}} \hspace*{-.3cm}
 \lp\frac{\innerp{\xi}{N_P(F_{1})}}{\innerp{\xi}{N_{F_1}(G)}} + \frac{\innerp{\xi}{N_P(F_{2})}}{\innerp{\xi}{N_{F_2}(G)}}\rp 
 \frac{e^{-2\pi i \innerp{\xi}{t \bar{x}_G}}}{\innerp{\xi}{\xi}} \hat{\phi}_\epsilon(\xi).
\end{equation*}

Next we treat each of these terms separately. The sum in $b_j(G;t)$ is simpler and is dealt with a direct application of Theorem~\ref{thm:lattice-sum}, which is in fact an application of Poisson summation. The sum in $c(G;t)$ takes more work and, after some preparation, is also dealt with 
the help of Theorem~\ref{thm:lattice-sum} (this time it is a $2$-dimensional lattice sum minus two lines) and in the end we recognize the occurrence of a Dedekind-Rademacher sum on each $(d-2)$-dimensional face of $P$.

\subsubsection{Computation of $b_j(G;t)$}

Let $j = 1$ and $m = 2$,  or $j = 2$ and  $m = 1$.  To compute $b_j(G;t)$, write $\xi \in \Lambda_{F_m}$ as $\xi = r v_{F_m}$ with $r \in \Z$:
\begin{align*}
b_j(G;t) = \lim_{\epsilon \to 0^+}\frac{\det(\Lambda_G)}{(-2\pi i)^2} \sum_{r \in \Z \setminus \{0\}} \frac{\innerp{v_{F_m}}{N_P(F_{j})} e^{-2\pi i r \innerp{v_{F_m}}{t \bar{x}_G}}}{\innerp{v_{F_m}}{N_{F_j}(G)}\|v_{F_m}\|^2 r^2} \hat{\phi}_\epsilon(r),
\end{align*}
where we use that $\hat{\phi}_\epsilon(rv_{F_m}) = \hat{\phi}_{\epsilon\|v_{F_m}\|^2}(r)$ and note that this can be replaced by~$\hat{\phi}_\epsilon(r)$ due to the limit in $\epsilon$.

Next, note that $\innerp{v_{F_{m}}}{N_P(F_j)} = \|v_{F_{m}}\|\innerp{N_P(F_m)}{N_P(F_j)} = -\|v_{F_{m}}\| c_G$ and that $\innerp{v_{F_{m}}}{N_{F_j}(G)} = \|v_{F_{m}}\|\innerp{N_P(F_m)}{N_{F_j}(G)} = \|v_{F_{m}}\| \sqrt{1 - c_G^2}$, so
\[\frac{\innerp{v_{F_{m}}}{N_P(F_j)}}{\innerp{v_{F_{m}}}{N_{F_j}(G)}} = \frac{-c_G}{\sqrt{1 - c_G^2}}.\]
We substitute this and recognize the $1$-dimensional sum $L_{\Z}((1), (2); \innerp{v_{F_m}}{t\bar{x}_G})$:
\[
b_j(G;t) = \frac{-c_G\det(\Lambda_G)}{\sqrt{1-c_G^2}\|v_{F_m}\|^2} \lim_{\epsilon \to 0^+}
\frac{1}{(2\pi i)^2} \sum_{r \in \Z \setminus \{0\}} 
\frac{e^{-2\pi i r \innerp{v_{F_m}}{t\bar{x}_G}}}{r^2}\hat{\phi}_\epsilon(r).
\]

Let $I$ be the interval $I := [\innerp{v_{F_m}}{t\bar{x}_G}, \innerp{v_{F_m}}{t\bar{x}_G} + 1]$ and apply Theorem~\ref{thm:lattice-sum}:
\[
b_j(G;t) = \frac{c_G\det(\Lambda_G)}{2\sqrt{1-c_G^2}\|v_{F_m}\|^2} \sum_{n \in \Z \cap I}B_2\big(n - \innerp{v_{F_m}}{t\bar{x}_G}\big)\omega_{I}(n).
\]

Depending on $\innerp{v_{F_m}}{t\bar{x}_G}$ being an integer or not, the sum may have one or two terms. In either case, since $B_2(0) = B_2(1)$ and since $\overline{B}_2$ is an even function,
\begin{equation*}
 b_j(G;t)= \frac{c_G\det(\Lambda_G)}{2\sqrt{1-c_G^2}\|v_{F_m}\|^2} \overline{B}_2\big(\innerp{v_{F_m}}{\bar{x}_G}t\big).
\end{equation*}
Recalling $\det(\Lambda_G) = |\det(v_{F_1}, v_{F_2})|/k$ (Lemma~\ref{lem:hk}), we get
\begin{equation}\label{eq:bjGt}
 b_j(G;t)= \frac{c_G\|v_{F_j}\|}{2k\|v_{F_m}\|} \overline{B}_2\big(\innerp{v_{F_m}}{\bar{x}_G}t\big).
\end{equation}

\subsubsection{Computation of $c(G;t)$}

The expression
\[ \frac{1}{\innerp{\xi}{\xi}}\lp\frac{\innerp{\xi}{N_P(F_{1})}}{\innerp{\xi}{N_{F_1}(G)}} + \frac{\innerp{\xi}{N_P(F_{2})}}{\innerp{\xi}{N_{F_2}(G)}}\rp \]
becomes simpler if we write $N_P(F_1)$, $N_P(F_2)$, and $\xi$ in terms of $N_{F_1}(G)$ and $N_{F_2}(G)$. 

From~\eqref{eq:unit-vectors}, we obtain
\begin{align}\label{eq:npf}
\begin{aligned}
N_P(F_1) &= \frac{-c_G}{\sqrt{1-c_G^2}} N_{F_1}(G) + \frac{1}{\sqrt{1-c_G^2}}N_{F_2}(G), \text{ and}\\
N_{P}(F_2) &= \frac{1}{\sqrt{1-c_G^2}} N_{F_1}(G) + \frac{-c_G}{\sqrt{1-c_G^2}}N_{F_2}(G).
\end{aligned}
\end{align}

To write $\xi \in \Lambda_G$ as a combination of $N_{F_1}(G)$ and $N_{F_2}(G)$, write $\xi = A N_{F_1}(G) + B N_{F_2}(G)$, take inner-products with $N_{F_1}(G)$ and $N_{F_2}(G)$ and solve a linear system to obtain:
\begin{multline}\label{eq:xi-npfj}
\xi = \lp \frac{\innerp{\xi}{N_{F_1}(G)} - c_G\innerp{\xi}{N_{F_2}(G)}}{1-c_G^2}\rp N_{F_1}(G)\\
+ \lp\frac{-c_G \innerp{\xi}{N_{F_1}(G)} + \innerp{\xi}{N_{F_2}(G)}}{1-c_G^2}\rp N_{F_2}(G).
\end{multline}

Next we add the two fractions
\[
\frac{\innerp{\xi}{N_P(F_{1})}}{\innerp{\xi}{N_{F_1}(G)}} + \frac{\innerp{\xi}{N_P(F_{2})}}{\innerp{\xi}{N_{F_2}(G)}}
= \frac{\innerp{\xi}{N_{F_2}(G)}\innerp{\xi}{N_P(F_{1})} + \innerp{\xi}{N_{F_1}(G)}\innerp{\xi}{N_P(F_{2})}}{\innerp{\xi}{N_{F_1}(G)}\innerp{\xi}{N_{F_2}(G)}},
\]
and substitute~\eqref{eq:npf},
\[
 = \frac{\innerp{\xi}{N_{F_1}(G)}^2 + 2c_G\innerp{\xi}{N_{F_1}(G)}\innerp{\xi}{N_{F_2}(G)} + \innerp{\xi}{N_{F_2}(G)}^2}{\sqrt{1-c_G^2}\innerp{\xi}{N_{F_1}(G)}\innerp{\xi}{N_{F_2}(G)}}.
\]
Substituting~\eqref{eq:xi-npfj} into $\innerp{\xi}{\xi}$, we get that the numerator of the last expression is $(1-c_G^2)\innerp{\xi}{\xi}$, hence
\[ \frac{1}{\innerp{\xi}{\xi}}\lp\frac{\innerp{\xi}{N_P(F_{1})}}{\innerp{\xi}{N_{F_1}(G)}} + \frac{\innerp{\xi}{N_P(F_{2})}}{\innerp{\xi}{N_{F_2}(G)}}\rp 
=  \frac{\sqrt{1-c_G^2}}{\innerp{\xi}{N_{F_1}(G)}\innerp{\xi}{N_{F_2}(G)}}.
\]

Substituting this into the definition of $c(G;t)$,
\[
c(G;t) = \lim_{\epsilon \to 0^+} \frac{\sqrt{1-c_G^2}\det(\Lambda_G)}{(-2\pi i)^2}\sum_{\xi \in \Lambda_G \setminus (\Lambda_{F_1} \cup \Lambda_{F_2})} \frac{e^{-2\pi i \innerp{\xi}{t\bar{x}_G}}}{\innerp{\xi}{N_{F_1}(G)}\innerp{\xi}{N_{F_2}(G)}}\hat{\phi}_\epsilon(\xi).
\]

This expression is similar to $L_\Lambda(W,e;x)$, that was considered in Section~\ref{sec:gl-sums}, however to use it we scale $N_{F_1}(G)$ and $N_{F_2}(G)$ to $v_{F_1, G}$ and $v_{F_2, G}$ to have vectors in the lattice $\Lambda_G^*$. Let $W$ be the matrix with $v_{F_1, G}$ and $v_{F_2, G}$ as columns. Then
\[c(G;t) = \sqrt{1-c_G^2}\det(\Lambda_G)\|v_{F_1, G}\|\|v_{F_2,G}\| L_{\Lambda_G}(W,(1,1); t\bar{x}_G).\]

Applying Theorem~\ref{cor:lattice-sum2}, we get
\[
c(G;t) = \frac{\sqrt{1-c_G^2} \|v_{F_1, G}\|\|v_{F_2,G}\|}{\det(W^\trans W)^{1/2}} \sum_{n \in \Lambda_G^* \cap P_{W, t \bar{x}_G}} 
\Bcal_{1,1}\big(W^{+}(n - t\bar{x}_G)\big)\omega_{P_{W, t\bar{x}_G}}(n),
\]
where $P_{W, t\bar{x}_G} := t\bar{x}_G + W[0,1]^2$ and $W^{+} := (W^\trans W)^{-1}W^\trans$ is the pseudoinverse of~$W$. Noting that
\begin{multline*}
\det(W^\trans W) = \|v_{F_1,G}\|^2\|v_{F_2,G}\|^2 - \innerp{v_{F_1,G}}{v_{F_2,G}}^2 = \|v_{F_1,G}\|^2\|v_{F_2,G}\|^2(1 - c_G^2),
\end{multline*}
we get
\begin{equation}\label{eq:cGt}
c(G;t) = \sum_{n \in \Lambda_G^* \cap P_{W,t \bar{x}_G}} \Bcal_{1,1}\big(W^{+}(n - t\bar{x}_G)\big)\omega_{P_{W,t\bar{x}_G}}(n).
\end{equation}

We now treat separately the terms in the boundary and in the interior of $P_{W, t \bar{x}_G}$.

\subsubsection{Terms in the boundary of $P_{W, t \bar{x}_G}$}\label{sec:boundary-points}

Since $P_{W, t \bar{x}_G}$ is a $2$-dimensional parallelepiped, if $n \in \Lambda_G^* \cap \partial P_{W, t \bar{x}_G}$, then $n$ is either in an edge or is a vertex of it.

\begin{figure}[t]
\begin{tikzpicture}
\fill[blue!45] (0,0) -- (2,1) -- (3,3) -- (1,2) -- cycle;
\draw[very thick, blue!80!black, ->] (0,0) -- (2,1) node[below = 5, blue!40!black] {$v_{F_1,G}$};
\draw[very thick, blue!80!black, ->] (0,0) -- (1,2) node[below left = 1, blue!40!black] {$v_{F_2,G}$};

\fill[blue!45] (3.2,.8) -- (5.2,1.8) -- (6.2,3.8) -- (4.2,2.8) -- cycle;
\filldraw[blue!40!black, fill = white] (3.2,.8) circle [radius = 1pt];
\draw[very thick, blue!80!black] (3.5,.95) arc (26.6: 63.4: 0.34);
\node[below, blue!40!black] at (3.2,.8) {$t\bar{x}_G$};
\draw[thick, blue!80!black, ->] (4.4,.8) .. controls (4.4,1) and (3.8,1.4) .. (3.46,1.06);
\node[below, blue!40!black] at (4.5,.85) {$\frac{\arccos(c_G)}{2\pi}$};

\draw (-1.3, 0) -- (7.3, 0) (0, -1.3) -- (0, 4.3);
\foreach \x in {-1, ..., 7}
\foreach \y in {-1, ..., 4}
\fill[black!60] (\x, \y) circle [radius = 1pt];

\foreach \p in {(4,2), (5,2), (5,3), (0,0), (1,1), (2,1), (1,2), (2,2), (3,3)}
{\fill[blue!40!black] \p circle [radius = 1pt];}; 
\end{tikzpicture}
\caption{The parallelepiped $P_{W, t\bar{x}_G}$ and the solid angle at its vertex.}
\end{figure}

If it is in an edge, say $n = t \bar{x}_G + p v_{F_1,G}$, with $0 < p < 1$, then since $v_{F_2, G} \in \Lambda_G^*$, we have that $n + v_{F_2, G}$ is in the middle of the opposite edge. Since both solid angles are equal to $1/2$, $n$ contributes to the sum with $B_1(p)B_1(0)/2$ and $n + v_{F_2, G}$ contributes with $B_1(p)B_1(1)/2$. Since $B_1(0) = -B_1(1)$, both terms cancel each other in the sum. The same situation happens in the edges spanned by $v_{F_2, G}$. Hence there is no contribution from the points in the edges.

If $n$ is a vertex of $P_{t \bar{x}_G}$, then, since $v_{F_1, G}, v_{F_2,G} \in \Lambda_G^*$, all four vertices are points from $\Lambda_G^*$ and contribute to the sum. Since $B_1(0)B_1(0) = B_1(1)B_1(1) = 1/4$ and $B_1(0)B_1(1) = -1/4$, it rests to compute the solid angles at the vertices.

Since the unit vectors in the directions of $v_{F_1, G}, v_{F_2,G}$ are $N_{F_1}(G)$ and $N_{F_2}(G)$ and $\innerp{N_{F_1}(G)}{N_{F_2}(G)} = c_G$, we have that $\omega_{P_{W,t\bar{x}_G}}(t\bar{x}_G) = \arccos(c_G)/(2\pi)$ and the solid angle at the other vertex is $(\pi - \arccos(c_G))/(2\pi)$.

The contribution of the four vertices becomes
\[
\frac{2}{4}\frac{\arccos(c_G) - (\pi - \arccos(c_G))}{2\pi} 
= \frac{\arccos(c_G)}{2\pi} -\frac{1}{4}
= \omega_P(G) -\frac{1}{4}.
\]

Since the condition for having the four vertices in $\Lambda_G^*$ is $t \bar{x}_G \in \Lambda_G^*$, the boundary lattice points of $P_{W, t \bar{x}_G}$ contributes with
\begin{equation}\label{eq:boundary}
\sum_{n \in \Lambda_G^* \cap \partial P_{W,t \bar{x}_G}} \hspace*{-.3cm} \Bcal_{1,1}(W^{+}(n - t\bar{x}_G))\omega_{P_{W,t\bar{x}_G}}(n) 
=\lp \omega_P(G) -\frac{1}{4} \rp \bm{1}_{\Lambda_G^*}(t \bar{x}_G)
\end{equation}
to the sum~\eqref{eq:cGt}. Note that this is the only term where nontrivial solid angles actually appear.

\subsubsection{Terms in the interior of $P_{W, t \bar{x}_G}$}\label{sec:interior-points}

For the terms in sum~\eqref{eq:cGt} that are in the interior of $P_{W, t \bar{x}_G}$, we
introduce a basis for the lattice $\Lambda_G^*$ and write $n$ in terms of it to recognize a Dedekind-Rademacher sum, as defined in~\eqref{eq:dedekind-rademacher}. 

Since $v_{F_1,G}$ is a $\Lambda_G^*$-primitive vector, we can set $v_1 := v_{F_1,G}$ and find $v_2 \in \Lambda_G^*$ such that $\{v_1, v_2\}$ is a basis for the lattice $\Lambda_G^*$. 
Letting $V$ be the matrix with $v_1$ and $v_2$ as columns, we have that $W = VA$ with $A = \lp\begin{smallmatrix} 1 & h\\ 0 & k \end{smallmatrix}\rp$ and $h, k$ coprime integers. By the choice of~$v_2$, we may assume that $k$ is positive and $0 \leq h < k$.

Now make the change of variables $n \mapsto Vn'$, so that $n'$ lies in $V^{+}\Lambda_G^* = \Z^2$:
\[\sum_{n \in \Lambda_G^* \cap \innt(P_{W, t \bar{x}_G})} \Bcal_{1,1}(W^{+}(n - t\bar{x}_G))
= \sum_{n \in \Z^2 \cap \innt(V^{+}P_{W, t \bar{x}_G})} \Bcal_{1,1}(A^{-1}n - tW^{+}\bar{x}_G) .\]

We compute $A^{-1} = \lp\begin{smallmatrix} 1 & -h/k\\ 0 & 1/k \end{smallmatrix}\rp$. Also recalling that $W^{+}\bar{x}_G =: \lp\begin{smallmatrix}
x_1\\ x_2 \end{smallmatrix}\rp$ and noting that $n = \lp\begin{smallmatrix}
n_1\\ n_2 \end{smallmatrix}\rp \in \innt(V^{+}P_{W, t \bar{x}_G}) \Leftrightarrow A^{-1}n - tW^{+}\bar{x}_G \in (0,1)^2$, we have:
\[
\left\{\begin{matrix}
0 < n_1 - \frac{h}{k}n_2 - tx_1 < 1,\\ 
0 < \frac{n_2}{k} - t x_2 < 1
\end{matrix}\right.
\Leftrightarrow
\left\{\begin{matrix}
tx_1 + \frac{h}{k}n_2 < n_1 < tx_1 + \frac{h}{k}n_2 + 1,\\ 
tx_2 k < n_2 < tx_2 k + k.
\end{matrix}\right.
\]
Therefore $n_2$ varies over all residues modulo $k$ and for each $n_2$ we have only one integer $n_1$ (except in the boundary cases $tx_2k \in \Z$ and $tx_1 + \frac{h}{k}n_2 \in \Z$, however the following stays true, since $\overline{B}_1(x) = 0$ for $x \in \Z$). Thus,
\begin{align*}
 \sum_{n \in \Z^2 \cap \innt(V^{+}P_{W, t \bar{x}_G})} &\Bcal_{1,1}(A^{-1}n - tW^{+}\bar{x}_G)\\
 &= -\sum_{r\hspace*{-.2cm}\mod k} \overline{B}_1\lp \frac{r}{k} - tx_2\rp \overline{B}_1\lp \frac{h}{k}r + tx_1\rp\\
 &= -\sum_{r\hspace*{-.2cm}\mod k} \overline{B}_1\lp \frac{r-tkx_2}{k} \rp \overline{B}_1\lp h\frac{r - tkx_2}{k} + t(x_1 + h x_2)\rp\\
 &= -s\big(h,k; t(x_1 + hx_2), -tkx_2\big),
\end{align*}
where in the first equality we use that $\overline{B}_1(x)$ is periodic and odd. In the last equality we recognize~\eqref{eq:dedekind-rademacher}. Hence the interior lattice points of $P_{W, t \bar{x}_G}$ contributes with
\begin{equation}\label{eq:interior}
\sum_{n \in \Lambda_G^* \cap \innt(P_{W, t \bar{x}_G})}\hspace*{-.5cm} \Bcal_{1,1}(W^{+}(n - t\bar{x}_G))\omega_{P_{W,t\bar{x}_G}}(n) 
= -s\big(h,k; (x_1 + hx_2)t, -kx_2t\big)
\end{equation}
to the sum~\eqref{eq:cGt}.
Finally, substituting~\eqref{eq:bjGt}, \eqref{eq:boundary}, and \eqref{eq:interior} into~\eqref{eq:ad2t}, we obtain the expression in the statement of Theorem~\ref{thm:ad2-formula}. \qed

\bigskip

Next, we prove Corollary \ref{cor:APt-int}.

\corAPtint* 

\begin{proof}
The formula from Theorem~\ref{thm:ad2-formula} for $a_{d-2}(t)$ is:
\begin{multline*}
a_{d-2}(t) = \hspace*{-.3cm}\sum_{\substack{G \subset P,\\ \dim G = d-2}}\hspace*{-.3cm} \vol^*(G)
\bigg[
\frac{c_G}{2k}
\bigg(\frac{\|v_{F_2}\|}{\|v_{F_1}\|} \overline{B}_2\big(\innerp{v_{F_1}}{\bar{x}_G} t \big) 
+ \frac{\|v_{F_1}\|}{\|v_{F_2}\|} \overline{B}_2\big(\innerp{v_{F_2}}{\bar{x}_G}t\big)\bigg)\\
+ \lp \omega_P(G) -\frac{1}{4}\rp \bm{1}_{\Lambda_G^*}\lp t\bar{x}_G \rp
- s\big(h,k; (x_1+hx_2)t, -kx_2t\big)
\bigg].
\end{multline*}

Since now we are assuming that $P$ is an integer polytope, all its faces have integer points and since $t$ is an integer, we have that $\innerp{v_F}{\bar{x}_G}t$ is an integer and thus both occurrences of $\overline{B}_2$ evaluate to $1/6$. The first term becomes
\[ \frac{c_G}{12k}\lp \frac{\|v_{F_2}\|}{\|v_{F_1}\|}+\frac{\|v_{F_1}\|}{\|v_{F_2}\|}\rp. \]

Letting $W$ be the matrix with $v_{F_1, G}$ and $v_{F_2, G}$ as columns and $V$ being the matrix with the lattice basis $v_1, v_2$ of $\Lambda_G^*$ as columns, recall that $\lp\begin{smallmatrix} x_1\\ x_2 \end{smallmatrix}\rp := W^{+}\bar{x}_G$ and $W^{+} = A^{-1}V^+$, where $A^{-1} = \lp\begin{smallmatrix} 1 & -h/k\\ 0 & 1/k \end{smallmatrix}\rp$. Since $\bar{x}_G = \Proj_{\lin(G)^\perp}(x_G)$ and $x_G$ can be chosen as an integer vector in the face $G$, $\bar{x}_G \in \Lambda_G^*$ (by Lemma~\ref{lm:dual-lattice}) and $V^+\bar{x}_G \in \Z^2$. Hence $\lp\begin{smallmatrix} x_1\\ x_2 \end{smallmatrix}\rp = \lp\begin{smallmatrix} 1 & -h/k\\ 0 & 1/k \end{smallmatrix}\rp \lp\begin{smallmatrix} n_1\\ n_2 \end{smallmatrix}\rp = \lp\begin{smallmatrix} n_1 - hn_2/k\\ n_2/k \end{smallmatrix}\rp$. Thus $x_1 + hx_2 = n_1 \in \Z$ and $kx_2 = n_2 \in \Z$ so the Dedekind-Radamacher sum $s\big(h,k;(x_1+hx_2)t,-kx_2t\big)$ becomes the Dedekind sum $s(h,k)$. Similarly, since $t\bar{x}_G \in \Lambda_G^*$, $\bm{1}_{\Lambda_G^*}\lp t\bar{x}_G \rp$ evaluates to $1$. 
\end{proof}

\bigskip
\section{Obtaining the Ehrhart quasi-coefficients $e_{d-1}(t)$ and $e_{d-2}(t)$}~\label{sec:Ehrhart}

In this section we show how the Ehrhart quasi-polynomial can be obtained from the solid angle sum quasi-polynomial by means of a limit process and we show that this relation also extends to the quasi-coefficients. As a result we obtain local formulas for the quasi-coefficients $e_{d-1}(t)$ and $e_{d-2}(t)$ for all positive real values of~$t$. The technique used here is an adaptation of a method used by Barvinok~\cite{barvinok06} for a similar purpose, but instead of giving finite formulas, he focuses in determining the algorithmic complexity of computing $e_{d-k}(t)$ for a fixed $k$.

Since we are dealing with different polytopes in this section, we modify the notation and write $e_k(P;t)$ and $a_k(P;t)$ in place of $e_k(t)$ and $a_k(t)$ for the quasi-coefficients of $L_P(t)$ and $A_P(t)$ respectively.

Let $P, R \subset \R^d$ be $d$-dimensional rational polytopes. We introduce the {\bf shifted solid angle sum} 
\[A_{P,R}(t) := \sum_{x \in \Z^d} \omega_{tP+R}(x),\]
where the ``$+$'' stands for the Minkowski sum $P+R := \{x + y : x \in P, y \in R\}$. 
Since the function $\varphi(P) := \sum_{x \in \Z^d} \omega_{P+R}(x)$ is a valuation\footnote{ I.e., satisfies $\varphi(P) + \varphi(Q) = \varphi(P \cup Q) + \varphi(P \cap Q)$ whenever $P \cup Q$ is a polytope.} on rational polytopes, McMullen~\cite{mcmullen78} shows that this shifted solid angle sum can also be expressed as a quasi-polynomial
\[A_{P,R}(t) = a_d(P,R;t)t^d + a_{d-1}(P,R;t)t^{d-1} + \dots + a_0(P,R;t),\] 
with period dividing the denominator of $P$, and hence {\em does not depending} on $R$, for integer values of $t$. Moreover, this expression can be extended to real values of~$t$ in the same manner than with the Ehrhart and solid angle sum expressions (c.f. Linke~\cite[Theorem 1.2]{linke11}). Thus, if $m$ is the denominator of $P$, we have ${a_{k}(P,R;t + m)} = a_{k}(P,R;t)$ for all $0 \leq k \leq d$ and $t \in \R$, $t > 0$.

\begin{theorem}\label{thm:lim-Ap}
Let $P \subset \R^d$ be a $d$-dimensional rational polytope and $a \in \innt(P)$ be a rational vector. Then pointwise for any positive real $t$,
\begin{equation*}
L_P(t) = \lim_{\tau \to 0^+} A_{P, \tau (P-a)}(t).
\end{equation*}
Furthermore,
\[e_k(P;t) = \lim_{\tau \to 0^+} a_k(P, \tau (P-a); t)\]
pointwise for all $0 \leq k \leq d$ and positive real $t$.
\end{theorem}
\begin{proof}
Since $P-a$ is a polytope with the origin in its interior, for any $t$, $\tau > 0$ we have that $tP \subset tP + \tau(P-a)$. Further, since $\Z^d$ is discrete, for any fixed positive real $t$ and all sufficiently small $\tau$, 
\[|(tP + \tau(P-a)) \cap \Z^d| = |tP \cap \Z^d|\quad \text{and}\quad \partial(tP + \tau(P-a)) \cap \Z^d = \emptyset.\]
This establishes the first claim.

To see how the limit also holds for the quasi-coefficients, let $m$ be the denominator of $P$. Since $m$ is a period for both the Ehrhart and the shifted solid angle sum quasi-coefficients,
we have 
\[e_k(P, t + jm) = e_k(P, t)\quad \text{and}\quad a_k(P, \tau (P-a); t + jm) = a_k(P, \tau (P-a); t),\]
for any integer $j \geq 0$ and $0 \leq k \leq d$. Evaluating the equality for quasi-polynomials with $t, t+m, \dots, t + dm$, we get the $d+1$ equations
\begin{equation*}
\sum_{k = 0}^d e_k(P, t)(t+jm)^k = \lim_{\tau \to 0^+} \sum_{k = 0}^d a_k(P, \tau (P-a); t)(t+jm)^k,\quad \text{for } j = 0, \dots, d.
\end{equation*}
Since the Vandermonde matrix $\big((t+jm)^k\big)_{j,k = 0}^d$ is invertible, these equations imply the equality for the quasi-coefficients.
\end{proof}

Theorem~\ref{thm:lim-Ap} gives a formula for $e_k(t)$ in terms of the quasi-coefficients of the shifted solid angle sum, however in Theorems~\ref{thm:ad1-formula} and~\ref{thm:ad2-formula} we have formulas for the solid angle sum quasi-coefficients without the shift. 
Next we adapt the proof of Theorem~\ref{thm:ak-formula} (from Diaz, Le, and Robins~\cite{diaz16}) where instead of considering the solid angle sum of the polytope $P$, we now consider the solid angle sum of the perturbed polytope $P + \tau(P-a)$ and we show that in the limit as $\tau \to 0^+$ 
both 
$\lim_{\tau \to 0^+} a_k(P, \tau (P-a); t)$ and $\lim_{\tau \to 0^+} a_k(P + \tau(P-a);t)$ 
are in fact the same.   

\begin{lemma} \label{lem:unshift}
Let $P \subset \R^d$ be a $d$-dimensional rational polytope and $a \in \innt(P)$ be a rational vector. Then pointwise for any positive real $t$,
\[
\lim_{\tau \to 0^+} a_k(P, \tau (P-a); t) = \lim_{\tau \to 0^+} a_k(P + \tau(P-a);t).
\]
Hence by Theorem~\ref{thm:lim-Ap} both expressions are equal to the Ehrhart quasi-coefficient $e_k(P;t)$.
\end{lemma}
\begin{proof}
In this proof we follow closely the procedure from Diaz, Le, and Robins~\cite{diaz16}, revised in Section~\ref{sec:summary-diaz16}. For any $t, \tau > 0$ we write the shifted solid angle sum $A_{P, \tau (P-a)}(t)$ using~Lemma~\ref{thm:conv-series}, followed by Poisson summation (Theorem~\ref{thm:poisson}):
\begin{align*}
A_{P, \tau (P-a)}(t) &= \sum_{x \in \Z^d}\omega_{tP + \tau(P-a)}(x)\\
&= \lim_{\epsilon \to 0^+} \sum_{x \in \Z^d} (\bm{1}_{tP + \tau(P-a)} \ast \phi_{d, \epsilon})(x)\\
&= \lim_{\epsilon \to 0^+} \sum_{\xi \in \Z^d} \hat{\bm{1}}_{tP + \tau(P-a)}(\xi)\hat{\phi}_{\epsilon}(\xi)\\
&= (t+\tau)^d\lim_{\epsilon \to 0^+} \sum_{\xi \in \Z^d} e^{-2\pi i \innerp{\xi}{-\tau a}} \hat{\bm{1}}_{P} ((t+\tau)\xi) \hat{\phi}_{\epsilon}(\xi),
\end{align*}
in the last line we use $\hat{\bm{1}}_{tP + \tau(P-a)}(\xi) = (t+\tau)^d e^{-2\pi i \innerp{\xi}{-\tau a}} \hat{\bm{1}}_{P}\big((t+\tau)\xi\big)$, which can be proven by the change of variables $x \mapsto (t+\tau)x - \tau a$ in the integral.

Next we apply the combinatorial Stokes formula~\cite[Theorem 1]{diaz16} for $P$ and use the rational weights $\Rcal_T(\xi)$ defined in Section~\ref{sec:summary-diaz16}.
\begin{multline*}
A_{P, \tau (P-a)}(t) = (t+\tau)^d \lim_{\epsilon \to 0^+} \sum_{T} \sum_{\xi \in \Z^d \cap S(T)} \hspace{-.25cm}(t+\tau)^{-l(T)} \Rcal_T(\xi) e^{-2\pi i \innerp{\xi}{(t+\tau)x_{T} - \tau a}} \hat{\phi}_\epsilon(\xi),
\end{multline*}
where the outer sum is taken over all chains of $G_P$ and $x_T$ is any point from the last face of chain $T$.   Similarly as in Theorem~\ref{thm:ak-formula}, this leads to a formula for the coefficients of $A_{P, -\tau a}(t + \tau) $:
\[A_{P, \tau (P-a)}(t) := \sum_{x \in \Z^d} \omega_{tP + \tau (P-a)}(x) 
= A_{P, -\tau a}(t + \tau) = \sum_{k=0}^d a_k(P, -\tau a;t + \tau) (t+\tau)^k,\]
where
\[
a_k(P, -\tau a;t + \tau)  = \lim_{\epsilon \to 0^+}  \sum_{T :\, l(T) = d-k} \sum_{\xi \in \Z^d \cap S(T)}  \Rcal_T(\xi) e^{-2\pi i\innerp{\xi}{(t+\tau)x_{T} - \tau a}} \hat{\phi}_\epsilon(\xi).
\]
To get the quasi-coefficients $a_{k}(P,\tau(P-a); t)$, we expand $(t+\tau)^k$ and rearrange the terms:
\begin{align*}
A_{P, \tau (P-a)}(t) &= \sum_{l=0}^d a_l(P, -\tau a;t + \tau) (t+\tau)^l
= \sum_{l=0}^d a_l(P, -\tau a;t + \tau) \sum_{k = 0}^l\binom{l}{k}t^k\tau^{l-k}\\
&= \sum_{k=0}^d\lp \sum_{l = k}^d \binom{l}{k}a_l(P, -\tau a;t + \tau) \tau^{l-k}\rp t^k.
\end{align*}
Hence
\begin{equation}\label{eq:ak-al}
a_k(P, \tau(P-a); t) = \sum_{l = k}^d \binom{l}{k}a_l(P, -\tau a;t + \tau) \tau^{l-k}.
\end{equation}
Before considering the limit $\tau \to 0^+$, next we show that the quasi-coefficients $a_l(P, -\tau a;t + \tau)$ can be bounded for all $\tau < 1$ and $t > 0$. Indeed, let $0 < \tau < 1$ and replace $t$ by $t - \tau$ so that we just have $t$ in the argument. Let $m$ be the period of $P$, since $A_{P, -\tau a}(t)$ is a quasi-polynomial with period $m$, we may assume $0 < t \leq m$. Evaluate $A_{P, -\tau a}(t)$ replacing $t$ by $t, t + m, \dots, t+dm$ to obtain $d+1$ equations 
\[A_{P, -\tau a}(t + jm) = \sum_{l = 0}^d a_l(P, -\tau a; t)(t + jm)^l\quad \text{for } 0 \leq j \leq d.\] 
Since the interpolation which sends the $d+1$ values $\big(A_{P, -\tau a}(t + jm)\big)_{j=0}^d$ to the coefficients $\big(a_l(P, -\tau a; t)\big)_{l = 0}^d$ is a linear transformation with matrix equal to the inverse of $\big((t+jm)^l\big)_{j,l = 0}^d$ and since its norm is a continuous function on $t$, it can be bounded for $0 \leq t \leq m$. Furthermore, the value $A_{P, -\tau a}(t + dm)$ is bounded for $\tau < 1$ and $0 < t \leq m$, thus the coefficients $a_l(P, -\tau a; t)$ are also bounded, as we claimed.

Now we fix a $t > 0$ and consider the limit $\tau \to 0^+$ in \eqref{eq:ak-al}. Since $|a_l(P, -\tau a;t + \tau)|$ is bounded independently on $t$ and $\tau < 1$, all terms with $l > k$ vanish as $\tau \to 0^+$ and we get
\begin{align*}
\lim_{\tau \to 0^+} a_{k}(P, \tau (P&-a); t) = \lim_{\tau \to 0^+} a_k(P, -\tau a;t + \tau)\\
&= \lim_{\tau \to 0^+} \lim_{\epsilon \to 0^+} \hspace{-.1cm} \sum_{T :\, l(T) = d-k} \sum_{\xi \in \Z^d \cap S(T)} \hspace{-.1cm} \Rcal_T(\xi) e^{-2\pi i\innerp{\xi}{(t+\tau)x_{T} - \tau a}} \hat{\phi}_\epsilon(\xi)\\
&= \lim_{\tau \to 0^+} \lim_{\epsilon \to 0^+} \hspace{-.1cm} \sum_{T :\, l(T) = d-k} \sum_{\xi \in \Z^d \cap S(T)} \hspace{-.1cm} \Rcal_T(\xi) e^{-2\pi i\innerp{t\xi}{(1+\tau)x_{T} - \tau a}} \hat{\phi}_\epsilon(\xi),
\end{align*}
where in the last step we make the change of variables $\tau \mapsto t\tau$ in the limit.

On the other hand, we may compute a expression for $a_{k}(P + \tau(P-a);t)$ using the original formula from Theorem~\ref{thm:ak-formula}, but with the polytope $P + \tau(P-a)$ instead of the polytope $P$. The chains of both polytopes can be identified, since the transformation $P \mapsto P + \tau(P-a)$ is a dilation followed by a translation. The rational weight gets multiplied by $(1+\tau)^{d - l(T)}$ due to the dilation of the faces and the fact that the weights $W_{(F_{j-1},F_j)}(\xi)$ only depend on the cone of feasible directions $\fcone(F_{j-1},F_j)$. The exponential weight becomes $e^{-2\pi i \innerp{\xi}{x_T + \tau(x_T - a})}$, thus

\begin{align*}
\lim_{\tau \to 0^+} a_k&(P + \tau(P-a); t)\\ 
&= \lim_{\tau \to 0^+} \lim_{\epsilon \to 0^+} \hspace{-.1cm} \sum_{T :\, l(T) = d-k} \sum_{\xi \in \Z^d \cap S(T)} \hspace{-.1cm} (1 + \tau)^{d-k} \Rcal_T(\xi) e^{-2\pi i\innerp{t\xi}{(1+\tau)x_{T} - \tau a}} \hat{\phi}_\epsilon(\xi)\\
&= \lim_{\tau \to 0^+} \lim_{\epsilon \to 0^+} \hspace{-.1cm} \sum_{T :\, l(T) = d-k} \sum_{\xi \in \Z^d \cap S(T)} \hspace{-.1cm} \Rcal_T(\xi) e^{-2\pi i\innerp{t\xi}{(1+\tau)x_{T} - \tau a}} \hat{\phi}_\epsilon(\xi),
\end{align*}
where we simply have taken the factor $(1+\tau)^{d-k}$ out and used the product rule of limits. The lemma follows since we obtained the same formula for both limits.
\end{proof}

With Lemma~\ref{lem:unshift} and Theorems~\ref{thm:ad1-formula} and~\ref{thm:ad2-formula}, we can produce formulas for $e_{d-1}(t)$ and $e_{d-2}(t)$ for all real $t > 0$. We recall the one-sided limits
\[ \overline{B}^+_1(x) := \lim_{\epsilon \to 0^+} \overline{B}_1(x + \epsilon) 
\quad\text{ and }\quad
\overline{B}^-_1(x) := \lim_{\epsilon \to 0^+} \overline{B}_1(x - \epsilon), \]
that differ from $\overline{B}_1(x)$ only at integer points ($\overline{B}^+_1(x) = - 1/2$ and $\overline{B}^-_1(x) = 1/2$ for $x \in \Z$).

\begin{theorem}\label{thm:ed1-formula}
Let $P$ be a full-dimensional rational polytope in $\R^d$. Then  for all positive real values of~$t$,
the codimension one quasi-coefficient of the Ehrhart function $L_P(t)$ has the following finite form:
\[ e_{d-1}(t) = -\hspace{-.3cm}\sum_{\substack{F \subset P,\\ \dim(F) = d-1}}\hspace{-.3cm} \vol^*(F)\overline{B}^+_1\big(\innerp{v_F}{x_F}t\big),\]
where $x_F$ is any point in $F$ and $v_F$ is the primitive integer vector in the direction of $N_P(F)$.
\end{theorem}
\begin{proof}
We have from Theorem~\ref{thm:ad1-formula} the formula for $a_{d-1}(t)$,
\[ a_{d-1}(P;t)\\ = -\hspace{-.3cm}\sum_{\substack{F \subset P,\\ \dim(F) = d-1}}\hspace{-.3cm} \vol^*(F)\overline{B}_1\big(\innerp{v_F}{x_F}t\big).\]
We use the formula from Lemma~\ref{lem:unshift} for $e_{d-1}(P;t)$ and observe that the effect of replacing the polytope $P$ by $P+\tau(P-a)$ is replace $F$ by $(1+\tau)F$ inside the relative volume and replace $x_F$ by $x_F + \tau(x_F - a)$. We get
\begin{align*}
e_{d-1}(P;t) &= \lim_{\tau \to 0^+} a_{d-1}(P + \tau(P-a);t)\\ &=
- \lim_{\tau \to 0^+} \hspace{-.3cm}\sum_{\substack{F \subset P,\\ \dim(F) = d-1}}\hspace{-.3cm} \vol^*\big((1+\tau)F\big) \overline{B}_1\big(\innerp{v_F}{x_F + \tau(x_F-a)}t\big)\\
&= -\hspace{-.3cm}\sum_{\substack{F \subset P,\\ \dim(F) = d-1}}\hspace{-.3cm} \vol^*(F)\overline{B}^{+}_1\big(\innerp{v_F}{x_F}t\big),
\end{align*}
where we have used that $\vol^*$ is continuous and $\innerp{v_F}{x_F-a} > 0$, since $v_F$ points outwards to $F$ and $a \in \innt(P)$.
\end{proof}

When $P$ is an integer polytope and $t$ is an integer, the formula from Theorem~\ref{thm:ed1-formula} simplifies to the classical formula
\[
 e_{d-1} = \frac{1}{2}\hspace{-.3cm} \sum_{\substack{F \subset P,\\ \dim(F) = d-1}}\hspace{-.3cm} \vol^*(F).
\]

The same technique can be applied to the computation of $e_{d-2}(t)$.

\thmedtwoformula* 
\begin{proof}
Once more we use the formula from Lemma~\ref{lem:unshift}, this time with the formula from Theorem~\ref{thm:ad2-formula} for $a_{d-2}(t)$:
\begin{multline*}
a_{d-2}(P;t) = \hspace*{-.3cm}\sum_{\substack{G \subset P,\\ \dim G = d-2}}\hspace*{-.3cm} \vol^*(G)
\bigg[
\frac{c_G}{2k}
\bigg(\frac{\|v_{F_2}\|}{\|v_{F_1}\|} \overline{B}_2\big(\innerp{v_{F_1}}{\bar{x}_G} t \big) 
+ \frac{\|v_{F_1}\|}{\|v_{F_2}\|} \overline{B}_2\big(\innerp{v_{F_2}}{\bar{x}_G}t\big)\bigg)\\
+ \lp \omega_P(G) -\frac{1}{4}\rp \bm{1}_{\Lambda_G^*}\lp t\bar{x}_G \rp
- s\big(h,k; (x_1+hx_2)t, -kx_2t\big)
\bigg].
\end{multline*}
The effect of replacing the polytope $P$ by $P+\tau(P-a)$ is scaling the relative volume of $G$ by $(1+\tau)^{d-2}$ and replace $\bar{x}_G$ by $\bar{x}_G + \tau(\bar{x}_G - \bar{a})$, where $\bar{a} = \Proj_{\lin(G)^\perp}(a)$. Recall that $x_1$ and $x_2$ are defined as the coordinates of $\bar{x}_G$ as a linear combination of $v_{F_1, G}$ and $v_{F_2, G}$, so  letting $W$ be the matrix with $v_{F_1,G}$ and $v_{F_2,G}$ as columns and $W^+ := (W^\trans W)^{-1}W^\trans$ being its pseudoinverse, we have $\lp\begin{smallmatrix} x_1\\ x_2 \end{smallmatrix}\rp = W^{+}\bar{x}_G$ and we replace it by $\lp\begin{smallmatrix} x_1\\ x_2 \end{smallmatrix}\rp + \tau W^{+}(\bar{x}_G - \bar{a})$. Note that by the orientation of $v_{F_1,G}$, $v_{F_2,G}$, and $\bar{x}_G - \bar{a}$, the vector $W^{+}(\bar{x}_G - \bar{a})$ has positive entries (see Figure~\ref{fig:rel-orient}).

To compute $\lim_{\tau \to 0^+} a_{d-2}(P + \tau(P-a);t)$, note that $\overline{B}_2$ is continuous, so we can replace $\tau$ by $0$ in it. $\Lambda_G^*$ is discrete and $\bar{x}_G - \bar{a} \neq 0$, so $\bm{1}_{\Lambda_G^*}\big( t(\bar{x}_G + \tau(\bar{x}_G-\bar{a}))\big) = 0$ for all sufficiently small $\tau$. To analyze the limit in the Dedekind-Rademacher sum, denote $\lp\begin{smallmatrix} a_1\\ a_2 \end{smallmatrix}\rp := W^{+}(\bar{x}_G - \bar{a})$ so:
\begin{align*}
 \lim_{\tau \to 0^+} s\big(h,k; (x_1+&\tau a_1)t + h(x_2+\tau a_2)t, -k(x_2+\tau a_2)t\big)\\
 &= \lim_{\tau \to 0^+}\sum_{r\hspace*{-.2cm}\mod k} \overline{B}_1\lp \frac{r}{k} - tx_2 -t\tau a_2\rp \overline{B}_1\lp \frac{h}{k}r + tx_1 + t\tau a_1\rp\\
 &= \sum_{r\hspace*{-.2cm}\mod k} \overline{B}^{-}_1\lp \frac{r}{k} - tx_2\rp \overline{B}^{+}_1\lp \frac{h}{k}r + tx_1\rp.
\end{align*}
Using the identities $\overline{B}^+_1(x) = \overline{B}_1(x) - \frac{1}{2}\bm{1}_{\Z}(x)$ and $\overline{B}^-_1(x) = \overline{B}_1(x) + \frac{1}{2}\bm{1}_{\Z}(x)$, we may rewrite it as
\begin{multline*}
= s(h,k;(x_1+hx_2)t,-kx_2t)
-\frac{1}{2}\sum_{r\hspace*{-.2cm}\mod k} \bm{1}_{\Z}\lp \frac{h}{k}r + tx_1 \rp \overline{B}_1\lp \frac{r}{k} - tx_2\rp\\
+\frac{1}{2}\sum_{r\hspace*{-.2cm}\mod k} \bm{1}_{\Z}\lp \frac{r}{k} - tx_2 \rp \overline{B}^{+}_1\lp \frac{h}{k}r + tx_1\rp.
\end{multline*}
Note that $hr/k + tx_1$ is an integer if and only if $tkx_1$ is an integer and $r \equiv -h^{-1}kx_1t \mod k$, where $h^{-1}$ denotes an integer satisfying $h^{-1}h \equiv 1 \mod k$ (in case $k = 1$ and $h = 0$, we take $h^{-1} = 1$). So the first sum becomes $\frac{1}{2} \bm{1}_{\Z}\lp kx_1t\rp \overline{B}_1\big( (h^{-1}x_1 +x_2)t \big)$. Similarly, $r/k - tx_2$ is an integer if and only if $tkx_2$ is an integer and $r \equiv tkx_2 \mod k$, so the second sum becomes $\frac{1}{2} \bm{1}_{\Z}(kx_2t)\overline{B}^+_1\big((x_1+hx_2)t\big)$.

Putting all this together, we get the desired formula for $e_{d-2}(t)$.
\end{proof}

When $P$ is an integer polytope and $t$ is an integer, the formula from Theorem~\ref{thm:ed2-formula} simplifies. Similarly to Corollary~\ref{cor:APt-int}, we have:

\coredtwoformula* 

Pommersheim found a very similar formula for $e_{d-2}$~\cite[Theorem 4]{pommersheim93}, where it is assumed $d = 3$ and $P$ an integer tetrahedra. The formula there is not a local formula though, since it is given in terms of the relative volumes of the facets of~$P$. The direct comparison of both formulas immediately gives an identity valid for tetrahedra, as follows.

\begin{corollary}
Let $P \subset \R^3$ be an integer tetrahedra. Then the following identity holds
\[
\sum_{\substack{G \subset P,\\ \dim G = d-2}} \frac{\vol^*(G)}{k} \left[ 
\frac{\|v_{F_1}\|}{\|v_{F_2}\|}\lp c_G - \frac{\vol(F_1)}{3\vol(F_2)}\rp
+ \frac{\|v_{F_2}\|}{\|v_{F_1}\|}\lp c_G - \frac{\vol(F_2)}{3\vol(F_1)}\rp
\right] = 0.
\]
\end{corollary}

\bigskip
\section{Two examples in three dimensions}\label{sec:examples}

In this section we consider two examples to show in detail how the computations described in Section~\ref{sec:complexity} are performed in practice. With Theorems~\ref{thm:ad1-formula} and~\ref{thm:ed1-formula} we also have a formula for the codimension one quasi-coefficients and even without having a general formula for the codimension three quasi-coefficients, in these examples we fully compute the quasi-polynomials for all positive real $t$ using the knowledge of $A_P(t)$ and $L_P(t)$ in the interval $0 < t < 1$. We also make use of the third periodized Bernoulli polynomial $\overline{B}_3(t) := (t - \lfloor t \rfloor)^3 - \frac{3}{2}(t - \lfloor t \rfloor)^2 + \frac{1}{2}(t - \lfloor t \rfloor)$.

\begin{example}\label{ex:standard-simplex}
The first example is the standard simplex $\Delta := \conv\{(0,0,0)^\trans$, $(1,0,0)^\trans$, $(0,1,0)^\trans$, $(0,0,1)^\trans\}$, whose solid angle polynomial was computed by Beck and Robins~\cite[Example 13.3]{beck15} for integer values of $t$. Here we show that for all positive real values of $t$,
\begin{multline*}
A_{\Delta}(t) = \frac{1}{6}t^3 - \frac{1}{2}\overline{B}_1(t) t^2 + 
\lp \frac{1}{2}\overline{B}_2(t) + \lp\frac{3}{2\pi}\arccos\lp \frac{1}{\sqrt{3}}\rp - \frac{3}{4}\rp \bm{1}_\Z(t) + \frac{1}{4}\rp t\\
-\overline{B}_1(t)\lp \frac{1}{6}\overline{B}_2(t) + \frac{2}{9}\rp,
\end{multline*}
and 
\begin{multline*}
L_{\Delta}(t) = \frac{1}{6}t^3 +\lp-\frac{1}{2}\overline{B}_1^{+}(t) + \frac{3}{4}\rp t^2  + \lp \frac{1}{2}\overline{B}_2(t) - \frac{3}{2}\overline{B}_1^+(t) + 1\rp t\\ 
-\frac{1}{6}\overline{B}_3(t) + \frac{3}{4}\overline{B}_2(t) - \overline{B}_1^{+}(t) + \frac{3}{8}.
\end{multline*}
\end{example}

\begin{figure}[t]
\tdplotsetmaincoords{60}{110}
\begin{tikzpicture}[tdplot_main_coords, scale = .75]
\coordinate (o) at (0,0);
\coordinate (a) at (5,0);
\coordinate (b) at (0,4.7);
\coordinate (a2) at (5.8,0);
\coordinate (b2) at (0,5.15);

\begin{scope}[yshift=3.9cm]
\coordinate (c) at (0,0); 
\end{scope}
\begin{scope}[yshift=4.3cm]
\coordinate (c2) at (0,0); 
\end{scope}

\fill[gray, opacity = .4] (a) -- (b) -- (o) -- cycle;
\fill[gray, opacity = .2] (a) -- (c) -- (o) -- cycle;
\fill[gray, opacity = .6] (c) -- (b) -- (o) -- cycle;

\draw[->, thick] (a) -- (a2);
\draw[->, thick] (b) -- (b2);
\draw[->] (c) -- (c2);

\draw[thick, dashed] (o) -- (a)  (o) -- (b)  (o) -- (c);
\draw[thick] (c) -- (b) -- (a) -- cycle;

\node[left=2] at (a) {$\lp\begin{smallmatrix} 1 \\ 0\\ 0\end{smallmatrix}\rp$};
\node[below=2] at (b) {$\lp\begin{smallmatrix} 0 \\ 1\\ 0\end{smallmatrix}\rp$};
\node[right=2] at (c) {$\lp\begin{smallmatrix} 0 \\ 0\\ 1\end{smallmatrix}\rp$};
\node[right=2] at ($(o)!.5!(a)$) {$e_1$};
\node[below] at ($(o)!.4!(b)$) {$e_2$};
\node[right] at ($(c)!.6!(o)$) {$e_3$};
\node[left] at ($(c)!.5!(a)$) {$e_4$};
\node[right=2] at ($(c)!.5!(b)$) {$e_5$};
\node[below] at ($(a)!.5!(b)$) {$e_6$};
\end{tikzpicture}
\caption{The standard simplex $\Delta := \conv\{(0,0,0)^\trans$, $(1,0,0)^\trans$, $(0,1,0)^\trans$, $(0,0,1)^\trans\}$ and its edges.}
\end{figure}

\begin{proof}
This polytope has four facets with corresponding supporting inequalities $F_1: x_1 \geq 0$, $F_2: x_2 \geq 0$, $F_3: x_3 \geq 0$, and $F_4: x_1 + x_2 + x_3 \leq~1$. 

We know that $A_\Delta(t)$ and $L_\Delta(t)$ have quasi-polynomial expressions
\begin{align*}
A_\Delta(t) &= \vol(\Delta) t^3 + a_2(t)t^2 + a_1(t)t +a_0(t),\\
L_\Delta(t) &= \vol(\Delta) t^3 + e_2(t)t^2 + e_1(t)t +e_0(t),
\end{align*}
with quasi-coefficients having period $1$, $a_0(0) = a_2(0) = 0$ (due to the Macdonald's Reciprocity Theorem~\cite[Theorem 13.7]{beck15}) and $e_0(0) = 1$ (due to~\cite[Corollary 3.15]{beck15}). 

We have that $\vol(\Delta) = 1/6$ and we can compute $a_2(t)$ and $e_2(t)$ with Theorems~\ref{thm:ad1-formula} and~\ref{thm:ed1-formula}. Since $\overline{B}_1(0) = 0$ and $0 \in F_1, F_2, F_3$, using that $\vol^*(F_4) = 1/2$ and $v_{F_4} = (1,1,1)^\trans$, we get 
\[a_2(t) = -\frac{1}{2}\overline{B}_1(t).\]
Since $\overline{B}_1^{+}(0) = -1/2$ and $\vol^*(F_1) = \vol^*(F_2) = \vol^*(F_3) = 1/2$, we get
\[e_2(t) = -\frac{1}{2}\overline{B}_1^{+}(t) + \frac{3}{4} .\]

We use Theorems~\ref{thm:ad2-formula} and~\ref{thm:ed2-formula} together with the procedure described in Section~\ref{sec:complexity} to compute~$a_1(t)$ and~$e_1(t)$. Due to the symmetry of $\Delta$, we only have to consider two edges.

The edge $e_1$ has incident facets $F_2$ and $F_3$ and relative volume $\vol^*(e_1) = 1$. From the inequalities, we get $v_{F_2} = (0,-1,0)^\trans$ and $v_{F_3} = (0,0,-1)^\trans$. From their inner product, we have $c_{e_1} = 0$ and $\omega_{\Delta}(e_1) =~1/4$. Next we write $U = (v_{F_2}, v_{F_3}) = \lp\begin{smallmatrix} 0 & 0\\ -1 & 0\\ 0 & -1\end{smallmatrix}\rp$ and compute the projection onto $\lin(e_1)^\perp$, $P = U(U^\trans U)^{-1}U^\trans = \lp\begin{smallmatrix} 0 & 0 & 0\\ 0 & 0 & 1\\ 0 & 1 & 0\end{smallmatrix}\rp$. Inspecting its columns, we get the lattice basis $\{(0,  0, 1)^\trans$, $(0, 1, 0)^\trans\}$ for $\Lambda_{e_1}^*$. Computing $f_{2,3}$ with formula~\eqref{eq:fjk} we obtain $(0,0,-1)^\trans$ and thus $v_{F_2, e_1} = (0,0,-1)^\trans$, also $f_{3,2} = (0,-1,0)^\trans$, so $v_{F_3, e_1} = (0,-1,0)^\trans$. Hence we can make $v_1 = v_{F_2, e_1}$ and $v_2 = v_{F_3, e_1}$ so that $h = 0$ and $k = 1$. Letting $V = (v_1, v_2)$, we compute $\det(\Lambda_{e_1}) = \det(V^\trans V)^{-1/2} = 1$. Since $(0,0,0)^\trans \in e_1$, we get $\bar{x}_{e_1} = P(0,0,0)^\trans = (0,0,0)^\trans$, so $x_1 = x_2 = 0$. 
With this information, the contribution from edge $e_1$ (and also from $e_2$ and $e_3$) to the sum in Theorem~\ref{thm:ad2-formula} is $-s(0,1;0,0) = 0$ and to the sum in Theorem~\ref{thm:ed2-formula} is $1/4$.

\medskip

The edge $e_4$ has incident facets $F_2$ and $F_4$ and relative volume $\vol^*(e_4) =~1$. From the inequalities, we get $v_{F_2} = (0,-1,0)^\trans$ and $v_{F_4} = (1,1,1)^\trans$. From their inner product, we have $c_{e_4} = 1/\sqrt{3}$ and $\omega_{\Delta}(e_4) = \arccos\lp 1/\sqrt{3}\rp/(2\pi)$. Next we write $U = (v_{F_2}, v_{F_4}) = \lp\begin{smallmatrix} 0 & 1\\ -1 & 1\\ 0 & 1\end{smallmatrix}\rp$ and compute the projection onto $\lin(e_4)^\perp$, $P = U(U^\trans U)^{-1}U^\trans = \lp\begin{smallmatrix} 1/2\ & 0\ & 1/2\\ 0\ & 1\ & 0\\ 1/2\ & 0\ & 1/2\end{smallmatrix}\rp$. Inspecting its columns, we get the lattice basis $\{(1/2, 0, 1/2)^\trans,$ $(0, 1, 0)^\trans\}$ for $\Lambda_{e_4}^*$. Computing $f_{2,4}$ with formula~\eqref{eq:fjk} we obtain $(1,1,1)^\trans + (0,-1,0)^\trans = (1,0,1)^\trans$ and thus $v_{F_2, e_4} = (1/2,0,1/2)^\trans$, also $f_{4,2} = (0,-3,0)^\trans + (1,1,1)^\trans = (1,-2,1)^\trans$, so $v_{F_4, e_4} = (1/2,-1,1/2)^\trans$. Hence we can make $v_1 = v_{F_2, e_4}$ and $v_2 = v_{F_4, e_4}$ (check that all columns from $P$ can be obtained as integer combinations of $v_1$ and $v_2$) so that $h = 0$ and $k = 1$. Letting $V = (v_1, v_2)$, we compute $\det(\Lambda_{e_4}) = \det(V^\trans V)^{-1/2} = \sqrt{2}$. Since $(1,0,0)^\trans \in e_4$, we get $\bar{x}_{e_4} = P(1,0,0)^\trans = (1/2,0,1/2)^\trans$ and $x_1 = 1$, $x_2 = 0$.
With this information, the contribution from edge $e_4$ (and also from $e_5$ and $e_6$) to the sum in Theorem~\ref{thm:ad2-formula}~is
\begin{align*}
\frac{1}{2\sqrt{3}} &\lp \frac{\sqrt{3}}{1} \overline{B}_2(0) + \frac{1}{\sqrt{3}}\overline{B}_2(t) \rp + \lp\frac{1}{2\pi}\arccos\lp \frac{1}{\sqrt{3}}\rp - \frac{1}{4}\rp \bm{1}_\Z(t) - s(0,1;t,0)\\
&= \frac{1}{12} + \frac{1}{6}\overline{B}_2(t) + \lp\frac{1}{2\pi}\arccos\lp \frac{1}{\sqrt{3}}\rp - \frac{1}{4}\rp \bm{1}_\Z(t),
\end{align*}
and to the sum in Theorem~\ref{thm:ed2-formula},
\begin{align*}
\frac{1}{2\sqrt{3}} &\lp \frac{\sqrt{3}}{1} \overline{B}_2(0) + \frac{1}{\sqrt{3}}\overline{B}_2(t) \rp - s(0,1;t,0) -\frac{1}{2}\bm{1}_\Z(t) \overline{B}_1(t) - \frac{1}{2}\overline{B}_1^+(t)  \\
&= \frac{1}{12} + \frac{1}{6}\overline{B}_2(t) - \frac{1}{2}\overline{B}_1^+(t).
\end{align*}

Multiplying by three to take into account the three similar edges, the coefficient~$a_{1}(t)$ of $A_{\Delta}(t)$, for all positive $t \in \R$, is
\[a_{1}(t) = \frac{1}{2}\overline{B}_2(t) + \lp\frac{3}{2\pi}\arccos\lp \frac{1}{\sqrt{3}}\rp - \frac{3}{4}\rp \bm{1}_\Z(t) + \frac{1}{4}.\]
When $t \in \Z$, this becomes $3\arccos\lp  1/\sqrt{3}\rp/(2\pi) - 5/12$, as computed in Example~13.3 of Beck and Robins~\cite{beck15}. Similarly, the coefficient~$e_{1}(t)$ of $L_{\Delta}(t)$, for all positive $t \in \R$, is
\[e_{1}(t) = \frac{1}{2}\overline{B}_2(t) - \frac{3}{2}\overline{B}_1^+(t) + 1.\]

To compute $a_0(t)$, we observe that for $0 < t < 1$, the only integer point in $t\Delta$ is~$(0,0,0)^\trans$ and its solid angle is $1/8$, so $A_\Delta(t) = 1/8$. Hence, for $0 < t < 1$,
\begin{align*}
 a_0(t) &= \frac{1}{8} - \frac{1}{6}t^3 - a_2(t)t^2 - a_1(t)t
 = \frac{1}{8} - \frac{1}{6}t^3 + \frac{1}{2}\overline{B}_1(t) t^2 - \lp \frac{1}{4} + \frac{1}{2}\overline{B}_2(t)\rp t\\
 &= -\overline{B}_1(t)\lp \frac{1}{6}\overline{B}_2(t) + \frac{2}{9}\rp.
\end{align*}
Similarly for $e_0(t)$, we have $L_\Delta(t) = 1$ for $0 < t < 1$, so
\begin{align*}
 e_0(t) &= 1 - \frac{1}{6}t^3 - e_2(t)t^2 - e_1(t)t\\
 &= 1 - \frac{1}{6}t^3   -\lp-\frac{1}{2}\overline{B}_1^{+}(t) + \frac{3}{4}\rp t^2  - \lp 1 + \frac{1}{2}\overline{B}_2(t) - \frac{3}{2}\overline{B}_1^+(t)\rp t\\
 &= -\frac{1}{6}\overline{B}_3(t) + \frac{3}{4}\overline{B}_2(t) - \overline{B}_1^{+}(t) + \frac{3}{8}.\qedhere
\end{align*}
\end{proof}

\begin{example}\label{ex:order-simplex}
The second example is the order simplex $\vartriangleleft\ := \conv\{(0,0,0)^\trans$, $(1,0,0)^\trans$, $(1,1,0)^\trans$, $(1,1,1)^\trans\}$, that receives this name since it corresponds to the linear ordering $x_3 \leq x_2 \leq x_1$ and is interesting since it tiles the cube together with the reflections corresponding to the six permutations of its coordinates. Here we show that for all positive real values of $t$,
\begin{multline*}
 A_\vartriangleleft(t) = \frac{1}{6}t^3 - \frac{1}{2}\overline{B}_1(t)t^2 + \lp \frac{1}{24} + \frac{1}{2}\overline{B}_2(t) - \frac{1}{8}\bm{1}_\Z(t) \rp t -\overline{B}_1(t)\lp \frac{1}{6}\overline{B}_2(t) + \frac{1}{72} \rp,
\end{multline*}
and,
\begin{multline*}
 L_\vartriangleleft(t) = \frac{1}{6}t^3 + \lp -\frac{1}{2}\overline{B}^{+}_1(t) + \frac{3}{4} \rp t^2 + \lp \frac{1}{2}\overline{B}_2(t) -\frac{3}{2}\overline{B}_1^+(t) + 1\rp t\\ 
 -\frac{1}{6}\overline{B}_3(t) + \frac{3}{4}\overline{B}_2(t) - \overline{B}_1^{+}(t) + \frac{3}{8}.
\end{multline*}
\end{example}

\begin{figure}[t]
\tdplotsetmaincoords{60}{20}
\begin{tikzpicture}[tdplot_main_coords, scale = .75]

\coordinate (o) at (0,0);
\coordinate (a) at (4.5,0);
\coordinate (b) at (4.5,5);
\coordinate (c) at (0,5);
\coordinate (a2) at (5.1,0);
\coordinate (c2) at (0,6);

\begin{scope}[yshift=3.6cm]
\coordinate (d) at (0,0); 
\coordinate (e) at (4.5,0); 
\coordinate (f) at (4.5,5); 
\coordinate (g) at (0,5); 
\end{scope}

\begin{scope}[yshift=4.25cm]
\coordinate (d2) at (0,0);
\end{scope}

\draw[thin, ->] (a) -- (a2);
\draw[thin, ->] (d) -- (d2);
\draw[->] (c) -- (c2);

\fill[gray, opacity = .4] (a) -- (b) -- (f) -- cycle;
\fill[gray, opacity = .3] (o) -- (b) -- (f) -- cycle;
\fill[gray, opacity = .6] (o) -- (a) -- (b) -- cycle;

\draw[thick] (o) -- (a) -- (b) -- (f) -- (o) (a) -- (f);
\draw[thick, dashed] (o) -- (b);

\coordinate (cb) at (intersection cs: 
            first line={(b) -- (c)}, second line={(o) -- (f)});

\draw[ultra thin] (d) -- (o) -- (c) -- (g) -- (d) -- (e) -- (f) -- (g) (a) -- (e) (c) -- (cb);
\draw[ultra thin, dashed] (cb) -- (b);

\node[below] at (o) {$\lp\begin{smallmatrix} 0 \\ 0\\ 0\end{smallmatrix}\rp$};
\node[below] at (a) {$\lp\begin{smallmatrix} 1 \\ 0\\ 0\end{smallmatrix}\rp$};
\node[right] at (b) {$\lp\begin{smallmatrix} 1 \\ 1\\ 0\end{smallmatrix}\rp$};
\node[right] at (f) {$\lp\begin{smallmatrix} 1 \\ 1\\ 1\end{smallmatrix}\rp$};
\node[below] at ($(o)!.5!(a)$) {$e_1$};
\node[below] at ($(o)!.5!(b)$) {$e_2$};
\node[right] at ($(a)!.5!(b)$) {$e_3$};
\node[right] at ($(f)!.5!(a)$) {$e_4$};
\node[right] at ($(f)!.5!(b)$) {$e_6$};
\node[left=2] at ($(f)!.5!(o)$) {$e_5$};
\end{tikzpicture}
\caption{The order simplex $\vartriangleleft\ := \conv\{(0,0,0)^\trans$, $(1,0,0)^\trans$, $(1,1,0)^\trans$, $(1,1,1)^\trans\}$ from Example~\ref{ex:order-simplex} and its edges.}
\end{figure}

\begin{proof}
This polytope has four facets with corresponding supporting inequalities $F_1: x_3 \geq 0$, $F_2: x_3 - x_2 \leq 0$, $F_3: x_2 - x_1 \leq 0$, and $F_4: x_1 \leq 1$.

Again, we know that $A_\vartriangleleft(t)$ and $L_\vartriangleleft(t)$ have quasi-polynomial expressions 
\begin{align*}
A_\vartriangleleft(t) &= \vol(\vartriangleleft) t^3 + a_2(t)t^2 + a_1(t)t +a_0(t),\\
L_\vartriangleleft(t) &= \vol(\vartriangleleft) t^3 + e_2(t)t^2 + e_1(t)t +e_0(t),
\end{align*}
with quasi-coefficients having period $1$, $a_0(0) = a_2(0) = 0$ (due to the Macdonald's Reciprocity Theorem~\cite[Theorem 13.7]{beck15}) and $e_0(0) = 1$ (due to ~\cite[Corollary 3.15]{beck15}). 

We have that $\vol(\vartriangleleft) = 1/6$ and we can compute $a_2(t)$ and $e_2(t)$ with Theorems~\ref{thm:ad1-formula} and~\ref{thm:ed1-formula}. Obtaining
\[a_2(t) = -\frac{1}{2}\overline{B}_1(t),\]
and
\[e_2(t) = -\frac{1}{2}\overline{B}^{+}_1(t) + \frac{3}{4}.\]

We use Theorems~\ref{thm:ad2-formula} and~\ref{thm:ed2-formula} together with the procedure described in Section~\ref{sec:complexity} to compute~$a_1(t)$ and~$e_1(t)$. To avoid repetition, we skip the computation of the contribution from edges $e_1$, $e_2$, $e_3$, and $e_4$. All them contribute with $0$ to $a_1(t)$ and they contribute with $\frac{3}{8}$, $\frac{1}{4}$, $-\frac{1}{2}\overline{B}_1^+(t)$, and $-\frac{1}{2}\overline{B}_1^+(t)$ respectively to $e_1(t)$.

The edge $e_5$ has incident facets $F_2$ and $F_3$ and relative volume $\vol^*(e_5) = 1$. From the inequalities, we get $v_{F_2} = (0,-1,1)^\trans$ and $v_{F_3} = (-1,1,0)^\trans$. From their inner-product, we have $c_{e_5} = 1/2$ and $\omega_{\vartriangleleft}(e_5) = 1/6$. Next we write $U = (v_{F_2}, v_{F_3}) = \lp\begin{smallmatrix} 0 & -1\\ -1 & 1\\ 1 & 0\end{smallmatrix}\rp$ and compute the projection onto $\lin(e_5)^\perp$, $P = U(U^\trans U)^{-1}U^\trans$ $= \lp\begin{smallmatrix} 2/3\ & -1/3\ & -1/3\\ -1/3\ & 2/3\ & -1/3\\ -1/3\ & -1/3\ & 2/3 \end{smallmatrix}\rp$. Inspecting its columns, we get the lattice basis $\{(2/3, -1/3$, $-1/3)^\trans$, $(-1/3, 2/3, -1/3)^\trans\}$ for $\Lambda_{e_5}^*$. Computing $f_{2,3}$ with~\eqref{eq:fjk} we obtain $(-2,1,1)^\trans$ and thus $v_{F_2, e_5} = (-2/3,1/3,1/3)^\trans$, also $f_{3,2} = (-1,-1,2)^\trans$, so $v_{F_3, e_5} = (-1/3$, $-1/3,2/3)^\trans$. Hence we can make $v_1 = v_{F_2, e_5}$ and $v_2 = v_{F_3, e_5}$ so that $h = 0$ and $k = 1$. Letting $V = (v_1, v_2)$, we compute $\det(\Lambda_{e_5}) = \det(V^\trans V)^{-1/2} = \sqrt{3}$. Since $(0,0,0)^\trans \in e_5$, we get $\bar{x}_{e_5} = P(0,0,0)^\trans = (0,0,0)^\trans$, so $x_1 = x_2 = 0$.
With this information, the contribution from edge $e_5$ to the sum in Theorem~\ref{thm:ad2-formula} is 
\[\frac{1}{4}\lp \frac{1}{6} + \frac{1}{6} \rp - \frac{1}{12} -s(0,1;0,0) = 0,\]
and to the sum in Theorem~\ref{thm:ed2-formula} is 
\[\frac{1}{4}\lp \frac{1}{6} + \frac{1}{6} \rp - s(0,1;0,0) + \frac{1}{4} = \frac{1}{3}.\]

\medskip

The edge $e_6$ has incident facets $F_3$ and $F_4$ and relative volume $\vol^*(e_6) = 1$. From the inequalities, we get $v_{F_3} = (-1,1,0)^\trans$ and $v_{F_4} = (1,0,0)^\trans$. From their inner-product, we have $c_{e_6} = 1/\sqrt{2}$ and $\omega_{\vartriangleleft}(e_6) = 1/8$. Next we write $U = (v_{F_3}, v_{F_4}) = \lp\begin{smallmatrix} -1 & 1\\ 1 & 0\\ 0 & 0 \end{smallmatrix}\rp$ and compute the projection onto $\lin(e_6)^\perp$, $P = U(U^\trans U)^{-1}U^\trans$ $= \lp\begin{smallmatrix} 1 & 0 & 0\\ 0 & 1 & 0\\ 0 & 0 & 0 \end{smallmatrix}\rp$. Inspecting its columns, we get the lattice basis $\{(1, 1, 0)^\trans$, $(0, 1, 0)^\trans\}$ for $\Lambda_{e_6}^*$. Computing $f_{3,4}$ with~\eqref{eq:fjk} we obtain $(1,1,0)^\trans$ and thus $v_{F_3, e_6} = (1,1,0)^\trans$, also $f_{4,3} = (0,1,0)^\trans$, so $v_{F_4, e_6} = (0,1,0)^\trans$. Hence we can make $v_1 = v_{F_3, e_6}$ and $v_2 = v_{F_4, e_6}$ so that $h = 0$ and $k = 1$. Letting $V = (v_1, v_2)$, we compute $\det(\Lambda_{e_6}) = \det(V^\trans V)^{-1/2} = 1$. Since $(1,1,0)^\trans \in e_6$, we get $\bar{x}_{e_6} = P(1,1,0)^\trans = (1,1,0)^\trans$, so $x_1 = 1$ and $x_2 = 0$.
With this information, the contribution from edge $e_6$ to the sum in Theorem~\ref{thm:ad2-formula} is 
\[
\frac{1}{2\sqrt{2}}\lp \sqrt{2}\, \overline{B}_2(t) + \frac{1}{6\sqrt{2}}\rp - \frac{1}{8}\bm{1}_\Z(t) -s(0,1;t,0)
= \frac{1}{2}\overline{B}_2(t) - \frac{1}{8}\bm{1}_\Z(t) + \frac{1}{24},\]
and to the sum in Theorem~\ref{thm:ed2-formula} is 
\begin{multline*}
\frac{1}{2\sqrt{2}}\lp \sqrt{2}\,\overline{B}_2(t) + \frac{1}{6\sqrt{2}}\rp -s(0,1;t,0) - \frac{1}{2}\bm{1}_\Z(t)\overline{B}_1(t) - \frac{1}{2}\overline{B}_1^+(t)\\ 
= \frac{1}{2}\overline{B}_2(t) -\frac{1}{2}\overline{B}_1^+(t) + \frac{1}{24}.
\end{multline*}

Therefore the coefficient $a_{1}(t)$ of $A_{\vartriangleleft}(t)$, for all positive $t \in \R$, is
\[a_{1}(t) = \frac{1}{2}\overline{B}_2(t) - \frac{1}{8}\bm{1}_\Z(t) + \frac{1}{24}.\]
When $t \in \Z$, this becomes $0$, as expected due to the fact that $\vartriangleleft$ tiles the space together with the simplices obtained by reflections across its facets. Similarly, the coefficient $e_1(t)$ of $A_{\vartriangleleft}(t)$, for all positive $t \in \R$, is
\begin{align*}
e_1(t) &= \frac{3}{8} + \frac{1}{4} - \overline{B}_1^+(t) + \frac{1}{3} + \frac{1}{2}\overline{B}_2(t) -\frac{1}{2}\overline{B}_1^+(t) + \frac{1}{24}
= \frac{1}{2}\overline{B}_2(t) -\frac{3}{2}\overline{B}_1^+(t) + 1.
\end{align*}

To compute $a_0(t)$, we observe that for $0 < t < 1$, the only integer point in $(t\!\vartriangleleft)$ is $(0,0,0)^\trans$ and its solid angle is $1/6 . 1/8$ (to see this, we use again that $\vartriangleleft$ together with six reflections tiles the cube), so $A_\vartriangleleft(t) = 1/48$. Hence, for $0 < t < 1$,
\begin{align*}
 a_0(t) &= \frac{1}{48} - \frac{1}{6}t^3 - a_2(t)t^2 - a_1(t)t
 = \frac{1}{48} - \frac{1}{6}t^3 + \frac{1}{2}\overline{B}_1(t) t^2 - \lp \frac{1}{24} + \frac{1}{2}\overline{B}_2(t)\rp t\\
 &= -\overline{B}_1(t)\lp \frac{1}{6}\overline{B}_2(t) + \frac{1}{72} \rp.
\end{align*}
Similarly for $e_0(t)$, we have $L_\vartriangleleft(t) = 1$ for $0 < t < 1$, so
\begin{align*}
 e_0(t) &= 1 - \frac{1}{6}t^3 - e_2(t)t^2 - e_1(t)t\\
 &= 1 - \frac{1}{6}t^3  - \lp -\frac{1}{2}\overline{B}^{+}_1(t) + \frac{3}{4} \rp t^2 - \lp \frac{1}{2}\overline{B}_2(t) -\frac{3}{2}\overline{B}_1^+(t) + 1\rp t\\
 &= -\frac{1}{6}\overline{B}_3(t) + \frac{3}{4}\overline{B}_2(t) - \overline{B}_1^{+}(t) + \frac{3}{8}.\qedhere
\end{align*}
\end{proof}

\medskip

\begin{remark}
Notice that albeit the polytopes in both examples have very different solid angle sum functions, they have the same Ehrhart function. This is not a surprise since the unimodular transformation $U = \lp\begin{smallmatrix} 1 & 1 & 1\\ 0 & 1 & 1\\ 0 & 0 & 1 \end{smallmatrix}\rp$ sends the standard simplex to the order simplex and since it maps $\Z^3$ to $\Z^3$, we indeed have $|(t \Delta) \cap \Z^3 | = |(t\! \vartriangleleft) \cap \Z^3 |$ for all positive real $t$. Since the transformation given by matrix $U$ is not orthogonal, it doesn't preserve solid angles though.
\end{remark}

\medskip
\section{Concrete polytopes and further remarks}\label{sec:concrete}

We introduce a family of polytopes, called concrete polytopes, which come up naturally in our context, and in the context of multi-tiling. Consider Example~\ref{ex:order-simplex}, where we had a polytope $P$ whose solid angle sum was $A_P(t) = \vol(P)t^d$, for all positive integer values of $t$. More generally, as done by Brandolini, Colzani, Robins, and Travaglini~\cite{brandolini19}, we say that a polytope $P$ is \textbf{concrete}~if:
\begin{equation} \label{concrete.polytope}
A_P(t) = \vol(P) t^d, 
\end{equation}
for all positive integer values of $t$. Such polytopes are very special, because their discrete volume $A_P(t)$ matches exactly their continuous (Lebesgue) volume.
 
As another example, consider any integer polygon $P$ in $\R^2$.  It is then always true that $A_P(t) = \vol(P)t^2$, for all positive integer values of $t$, which is an equivalent formulation of {\em Pick's Theorem}.
 
The motivation for using the word `concrete' is borrowed from the title of the book ``Concrete Mathematics", where Graham, Knuth, and Patashnik mention that the word `concrete', which uses the first $3$ letters of `continuous', and the last~$5$ letters of `discrete',  embodies objects that are both ``continuous" and ``discrete".  

Another special family of concrete polytopes is the collection of \textbf{integer zonotopes} (see Lemma~\ref{cs.facets} below). Integer zonotopes are projections of cubes or, equivalently, integer polytopes whose faces (of all dimensions) are centrally symmetric (see e.g. Ziegler~\cite[Section 7.3]{ziegler95}). Alexandrov~\cite{alexandrov33}, and independently Shephard~\cite{shephard67}, proved the following fact.

\begin{lemma}[Alexandrov, Shephard]\label{cs1}
Let $P$ be any real, $d$-dimensional polytope, with $d \geq 3$.  If the facets of  $P$ 
are centrally symmetric, then $P$ is centrally symmetric.
\end{lemma}

The following statement appeared in \cite[Corollary 7.7]{barvinok99}, but we offer a proof here that is in the spirit of the current work.

\begin{lemma}[Barvinok]  \label{cs.facets}
Suppose $P$ is a $d$-dimensional integer polytope in $\R^d$ 
all of whose facets are centrally symmetric.  Then $P$ is a concrete polytope. 
\end{lemma}
\begin{proof}
We recall the formula for the solid angle polynomial $A_P(t)$ from Lemma~\ref{lm:APt-hat1P}:
\begin{equation}\label{eq:APtsum}
A_P(t) = \lim_{\epsilon \to 0^+} \sum_{\xi \in \Z^d} \hat{\bm{1}}_{tP}(\xi)e^{-\pi\epsilon\|\xi\|^2}.
\end{equation}

The Fourier transform of the indicator function of a polytope may be written as follows, after one application of the `combinatorial Stokes' formula (see \cite{diaz16}, equation (26)):
\begin{align}  
\hat{\bm{1}}_{tP}(\xi) = t^d \vol(P)[\xi =0]
+ \left( \frac{-1}{2 \pi i}  \right)      t^{d-1} 
\sum_{{\substack{F \subset P  \\ \dim F = d-1}}} 
 \frac{\langle \xi, N_P(F) \rangle }{\| \xi \|^2}      \hat{\bm{1}}_F       (t \xi) [\xi \not= 0],
\end{align}
where we sum over all facets $F$ of $P$.   Plugging this into~\eqref{eq:APtsum} we get
\begin{align}  \label{imaginary}
A_P(t)  - t^d \vol(P)
= \left( \frac{-1}{2 \pi i} \right) t^{d-1}
\lim_{\epsilon \to 0^+} \sum_{\xi \in \Z^d \setminus\{0\}}
     \frac{e^{-\pi\epsilon\|\xi\|^2}}{\| \xi \|^2} \sum_{{\substack{F \subset P  \\ \dim F = d-1}}}
\langle \xi, N_P(F) \rangle   \hat{\bm{1}}_F(t \xi) 
\end{align}
Thus, if we show that the latter sum over the facets vanishes, then we are done.

The assumption that all facets of $P$ are centrally symmetric implies that $P$ itself is also centrally symmetric, by Lemma~\ref{cs1}. 
We may therefore combine the facets of $P$ in pairs of opposite facets $F$ and $F'$. We know that $F' = F + c$, where $c$ is an integer vector, using the fact that the facets are centrally symmetric.

Therefore, since $N_P(F') = - N_P(F)$, we have
\begin{align*}
\langle \xi, N_P(F) \rangle  &\hat{\bm{1}}_{F}(t \xi) 
+ \langle \xi, -N_P(F) \rangle   \hat{\bm{1}}_{F+ c}(t \xi)\\
&= \langle \xi, N_P(F) \rangle   \hat{\bm{1}}_{F}(t \xi) 
- \langle \xi, N_P(F) \rangle   \hat{\bm{1}}_{F}(t \xi) 
e^{-2\pi i\langle   t\xi, c  \rangle} \\
&=  \langle \xi, N_P(F)\hat{\bm{1}}_{F}(t \xi) 
\big( 1  -  e^{-2\pi i\langle   t\xi, c  \rangle} \big) = 0,
\end{align*}
because $\langle t\xi, c \rangle \in \Z$ for $\xi \in \Z^d$ and $t \in \Z$.  We conclude that the entire right-hand side of \eqref{imaginary} vanishes, proving the lemma.
\end{proof}

Fourier analysis can also be used to give more general classes of polytopes that satisfy the formula $A_P(t) = \vol(P)t^d$, for positive integer values of $t$. A polytope $P$ is said to \textbf{$k$-tile} $\R^d$ (or \textbf{multi-tile $\R^d$ at level $k$}) by integer translations, if 
\begin{equation} \label{definition of multitiling}
\sum_{\lambda \in \Z^d}\bm{1}_P(x - \lambda) = k
\end{equation}
for every $x \notin \partial P + \Z^d$. Gravin, Robins, and Shiryaev~\cite[Theorem 6.1]{gravin12} gave a characterization of these polytopes in terms of solid angles.

\begin{theorem}[Gravin, Robins, Shiryaev]
A polytope $P$ $k$-tiles $\R^d$ by integer translations if and only if
\[\sum_{\lambda \in \Z^d} \omega_{P + v}(\lambda) = k,\]
for every $v \in \R^d$.
\end{theorem}

Note that the sum on the left is equal to $A_{P+v}(1)$, so this condition can be rephrased as asking for the function $P \mapsto A_P(1)$ to be invariant under all real translates of~$P$. 
To see how multi-tiling implies the concrete polytope property, note that since $f(x) :=\sum_{\lambda \in \Z^d}\bm{1}_P(x - \lambda)$ is periodic modulo $\Z^d$, it has a Fourier series (see e.g.,~\cite[Chapter VII, Theorem 2.4]{stein71})
$f(x) = \sum_{\xi \in \Z^d}\hat{\bm{1}}_P(\xi)e^{2\pi i \innerp{\xi}{x}},$
and so $P$ k-tiles by integer translations if and only if $\hat{\bm{1}}_P(\xi) = 0$ for all $\xi \in \Z^d \setminus \{0\}$, and $\hat{\bm{1}}_P(0) = k = \vol(P)$. By Lemma~\ref{lm:APt-hat1P}, we see that this implies $A_P(t) = \vol(P)t^d$ for all $t \in \Z$, $t > 0$.

Note that the order simplex in Example~\ref{ex:order-simplex} doesn't $k$-tile $\R^3$ by integer translations;  however, this simplex is still concrete. To produce more general concrete polytopes,  we introduce two new concepts: 

The \textbf{Hyperoctahedral group} $\Bd$ is the group of symmetries of the hypercube $[-1,1]^d$; all of its $2^d d!$ elements are simultaneously unimodular and orthogonal transformations, hence when an element of this group is applied to a polytope it preserves its solid angle polynomial. 

The \textbf{polytope group} $\mathscr{P}^d$ (cf.~\cite[Section 3.2]{lev18}) is the abelian group formally generated by the elements $[A]$ where $A$ runs through all sets in $\R^d$ which can be represented as the union of a finite number of polytopes with disjoint interiors and subject to the relations $[A] + [B] = [A \cup B]$ whenever $A$ and $B$ are two sets with disjoint interiors. Since any element $P \in \mathscr{P}^d$ can be uniquely represented as a finite sum $Q = \sum_j m_j[A_j]$ where $m_j$ are distinct nonzero integers and $A_j$ are sets with pairwise disjoint interiors, any additive function $\varphi$ defined on the set of polytopes in~$\R^d$ (such as the volume $A \mapsto \vol(A)$ or the indicator function $A \mapsto \bm{1}_A$ viewed as a function in $L^1(\R^d)$) can be uniquely extended to a function in $\mathscr{P}^d$ by linearity, that is, $\varphi(Q) := \sum_j m_j \varphi(A_j)$ for an element $Q$ written as above. With this extension, the definition of multi-tiling can also be extended to $\mathscr{P}^d$.

With these definitions, we may adapt the proof of (the forward direction of)~\cite[Theorem 4.1]{lev18} and prove the following more general sufficiency condition for the concrete polytope property.

\begin{theorem}\label{thm:poly-group}
If $P$ is a rational polytope in $\R^d$ such that $Q := \sum_{\gamma \in \Bd} [\gamma P]$ multi-tiles $\R^d$ by integer translations, then $A_P(t) = \vol(P)t^d$ for all positive integers $t$.
\end{theorem}
\begin{proof}
If $P$ is a rational polytope such that $Q = \sum_{\gamma \in \Bd} [\gamma P]$ $k$-tiles $\R^d$ by integer translations, then for a positive integer $t$ we also have that $Q_t := \sum_{\gamma \in \Bd} [\gamma (tP)]$ $(t^dk)$-tiles $\R^d$. Let $D := [0,1]^d$, then $Q_t - (t^dk)[D]$ tiles at level zero by integer translations and by~\cite[Proposition 3.4]{lev18} we can represent it as a finite sum $Q_t - (t^dk)[D] = \sum_j([B_j] - [B_j'])$ where for each $j$, $B_j$, $B_j'$ are polytopes such that $B_j'$ is obtained from $B_j$ by a translation along an integer vector, thus $A_{B_j}(1) = A_{B_j'}(1)$. Hence 
\[A_{Q_t}(1) = t^dk A_D(1) = t^dk = \vol(Q)t^d = |\Bd|\vol(P)t^d,\]
where we have used that if $Q$ $k$-tiles $\R^d$ by integer translations, then $k = \vol(Q)$ and that the action of $\Bd$ preserves volumes, thus $\vol(\gamma P) = \vol(P)$ for all $\gamma \in \Bd$. Also,
\[A_{Q_t}(1) = \sum_{\gamma \in \Bd}A_{\gamma(tP)}(1) = \sum_{\gamma \in \Bd}A_{tP}(1) = |\Bd| A_P(t). \qedhere\]
\end{proof}

\begin{example}
The simplex $\vartriangleleft$, which we used in Example~\ref{ex:order-simplex}, is now seen to satisfy 
 the condition of Theorem~\ref{thm:poly-group}, because  it tiles the cube together with the reflections corresponding to the six permutations of coordinates and these reflections are a subgroup of $\mathsf{B}_3$. Further, the simplex $\frac{1}{2}\!\vartriangleleft$ also satisfies the hypothesis of    Theorem~\ref{thm:poly-group}
 (because the orbit of $\frac{1}{2}\!\vartriangleleft$ under the action of $\mathsf{B}_3$ produces the cube $[-1/2, 1/2]^3$ that tiles the space by integer translations) and this is an example of a rational (and non-integer) polytope that has the concrete polytope property.  
 
As a side-note, this fact can also be seen in the expression given for $A_{\vartriangleleft}(t)$ in Example~\ref{ex:order-simplex}, verifying the fact that  $a_2(1/2) = a_1(1/2) = a_0(1/2) = 0$, and using the fact that all coefficients of the  quasi-polynomial have period $1$. 
\end{example}

It seems natural to ask whether multi-tiling is a necessary and sufficient condition for a polytope to be concrete, as follows.

\noindent
{\bf Question}.  \label{conj:multitiling}
Is a rational polytope $P \subset \R^d$ concrete if and only if 
\[
Q :=\sum_{\gamma \in \Bd} [\gamma P]
\]
multi-tiles $\R^d$ by integer translations?

This question has a negative answer, however, as very recently shown by  Garber and Pak~\cite{garber20}. They produced a counterexample in $\R^3$, based on the Dehn invariant of the direct sum of some tetrahedra and then extended it to $\R^d$. It remains an open question to give necessary and sufficient conditions for which $P$ is concrete. Perhaps a more general type of tiling is required.

\medskip
\noindent
{\bf Question.}
Suppose we know the solid angle quasi-polynomial $A_P(t)$, for all positive~$t$, but we also know that it is associated to a rational polytope $P$.  Can we recover $P$ completely, up to the action of the finite hyperoctahedral  group $\Bd$?

\medskip
\noindent
{\bf Question.}
Can the current theory be extended to all real polytopes?  

\bigskip
\bibliographystyle{plain}

\appendix

\section{Local formulas, SI-interpolators and alternative approaches}~\label{ap:ehrhart-to-solidangle}

In this appendix we briefly summarize the method developed by Berline and Vergne~\cite{berline07}, which produce local formulas for the Ehrhart quasi-coefficients based on a connection between exponential sums and integrals. We refer the reader to the original paper~\cite{berline07} for the formal definitions and detailed proofs and also to the subsequent works of Barvinok~\cite{barvinok08} and Garoufalidis and Pommersheim~\cite{garoufalidis12}, from which we borrow some concepts.
 
Consider the set $V$ of all convex polyhedra in $\R^d$.  
A {\bf valuation} in $V$ is any map~$\phi$ from $V$ to some vector space 
that enjoys the property $\phi(\emptyset) = 0$ and also
\[
\phi(P \cup Q) = \phi(P) + \phi(Q) - \phi(P\cap Q),
\]
for all $P, Q \in V$ such that $P \cup Q \in V$. In other words, valuations respect the inclusion-exclusion property enjoyed by polyhedra. 
Our goal here is to describe a valuation $\mu$ that associates to every polyhedra an analytic function on $\C^d$ which can be used to define a local formula for the Ehrhart polynomial~\cite[Corollary 30]{berline07}:
\[
|P \cap \Z^d| = \sum_{F \subseteq P} \vol^*(F) \mu\big(\tcone(P,F)\big)(0),
\]
and therefore for its coefficients
\[
e_k(P;t) = \sum_{\substack{F \subset P\\ \dim(F) = k}} \vol^*(F) \mu\big(\tcone(tP,tF)\big)(0).
\]

Here we focus on two particular valuations, called 
the \textbf{exponential integral}~$I(P)$ and the \textbf{exponential sum} $S(P)$,
which associates to each polyhedra $P \in V$ a meromorphic function on $\C^d$ that is zero on polyhedra that contains lines and such that for every $\xi \in \C^d$ which makes the right hand  side absolutely integrable or summable,
\[
I(P)(\xi) = \int_P e^{\innerp{\xi}{x}}\diff_P(x), \qquad S(P)(\xi) = \sum_{x \in P \cap \Z^d} e^{\innerp{\xi}{x}},
\]
where $\diff_P$ denotes the Lebesgue measure on the affine span of $P$ normalized so that $\det(\lin(P) \cap \Z^d) = 1$. Notice that $I(P)$ is the Fourier transform of $P$, up to a change of variables and an extension of domain.

Recall that if $F$ is a face of a polyhedron $P$, the tangent cone of $P$ at $F$ is 
\[
\tcone(P,F) := \{x + \lambda (y-x) : x \in F, y \in P, \lambda \geq 0\}.
\]
Berline and Vergne~\cite[Theorem 20]{berline07} proved the existence of a valuation $\mu$ which associates to every rational affine cone in $\R^d$ an analytic function such that for every rational polyhedron $P$, we have
\begin{equation}\label{eq:SI-interpolator}
S(P)(\xi) = \sum_{F \subseteq P} \mu\big(\tcone(P,F)\big)(\xi)I(F)(\xi),
\end{equation}
where the sum is taken over the set of all faces of $P$. The valuation $\mu$ is called a \textbf{SI-interpolator} by Garoufalidis and Pommersheim~\cite{garoufalidis12} and is uniquely defined up to a 
certain rule that extends functions initially defined on subspaces, to functions that exist on the entire space (e.g., orthogonal projection).

Now we assume that $P$ is a polytope (hence compact).  
Since $\mu\big(\tcone(P,F)\big)$ is an analytic function,  we may use its 
Taylor expansion at $0$ to define a differential operator for each face $F$ of $P$:
\begin{equation}\label{Differential operator}
D(P,F) := \mu\big(\tcone(P,F)\big)(\partial_x).
\end{equation}
This operator satisfies:
\begin{equation}
 D(P,F) e^{\innerp{\xi}{x}} = \mu\big(\tcone(P,F)\big)(\xi) e^{\innerp{\xi}{x}}.
\end{equation}
In other words, the valuations $\mu\big(\tcone(P,F)\big)(\xi)$ are eigenvalues of the differential 
operator  \eqref{Differential operator}. 

Now, for any polynomial $h(x)$ we may define an associated differential operator 
\[
D_h := h(\partial_{\xi}),
\]
which clearly satisfies
\[
 D_h e^{\innerp{\xi}{x}} = h(x) e^{\innerp{\xi}{x}}.
\]
Next, we take~\eqref{eq:SI-interpolator} and apply the definition of the operator 
$D(P,F)$ to get
\begin{align*}
S(P)(\xi) &= \sum_{F \subseteq P} \mu\big(\tcone(P,F)\big)(\xi)I(F)(\xi)\\
&= \sum_{F \subseteq P} \int_F \mu\big(\tcone(P,F)\big)(\xi)e^{\innerp{\xi}{x}}\diff_F(x)\\
&= \sum_{F \subseteq P} \int_F D(P,F)e^{\innerp{\xi}{x}}\diff_F(x).
\end{align*}
Applying $D_h$ to both sides,
\begin{align*}
D_hS(P)(\xi) &= \sum_{F \subseteq P} \int_F D(P,F)D_he^{\innerp{\xi}{x}}\diff_F(x)\\
\sum_{x \in P \cap \Z^d} h(x) e^{\innerp{\xi}{x}} &= \sum_{F \subseteq P} \int_F D(P,F) h(x)e^{\innerp{\xi}{x}}\diff_F(x).
\end{align*}
Evaluating the latter identity at $\xi = 0$, we get:
\begin{equation}\label{eq:euler-maclaurin}
\sum_{x \in P \cap \Z^d} h(x) = \sum_{F \subseteq P} \int_F D(P,F) h(x)\diff_F(x).
\end{equation}

Equation~\eqref{eq:euler-maclaurin} is called an Euler-Maclaurin summation formula since the sum 
on the right is expressed in terms of integrals taken over the faces of $P$,  of functions that depend only on local information along each face. 
Applying it to the constant function $h(x) = 1$ and noticing that the constant term of $D(P,F)$ is equal to $\mu\big(\tcone(P,F)\big)(0)$, we get a local formula for the Ehrhart polynomial~\cite[Corollary 30]{berline07}:
\[
 |P \cap \Z^d| = \sum_{F \subseteq P} \vol^*(F) \mu\big(\tcone(P,F)\big)(0),
\]
and thus for its coefficients
\[
 e_k(P;t) = \sum_{\substack{F \subset P\\ \dim(F) = k}} \vol^*(F) \mu\big(\tcone(tP,tF)\big)(0).
\]

Finally, we note that equation~\eqref{eq:SI-interpolator} together with $\mu({0})(\xi) = 1$ defines $\mu(K)$ recursively on the dimension of $K$ and this relation can be used to compute the coefficients for low dimensional cones~\cite[Proposition 31]{berline07}.

\bigskip

Once the formulas for the Ehrhart coefficients are obtained by the method outlined above, the solid angle sum of a polytope can also be obtained with the following formula
\begin{equation*}
A_P(t) = \sum_{F \subseteq P}\omega_P(F)L_{\innt(F)}(t) = \sum_{F \subseteq P}\omega_P(F)(-1)^{\dim(F)}L_{F}(-t).
\end{equation*}
This approach reverse the order of things we have taken in this paper. We take it here to complement Section~\ref{sec:Ehrhart} and show how the Ehrhart and solid angle sum expressions interrelate. Next we take the formulas from Theorems~\ref{thm:ed1-formula} and~\ref{thm:ed2-formula} as given and use them to recover Theorems~\ref{thm:ad1-formula} and~\ref{thm:ad2-formula}. 

Expanding the Ehrhart quasi-polynomials of each face and comparing quasi-coefficients, we get:
\begin{align*}
 \vol(P) &t^d + a_{d-1}(P;t)t^{d-1} + a_{d-2}(P;t)t^{d-2} + \dots\\
 &= \vol(P)t^d 
 + \Big( -e_{d-1}(P;-t) + \hspace{-.4cm}\sum_{\substack{F \subset P\\ \dim(F) = d-1}}\hspace{-.4cm}\frac{1}{2}e_{d-1}(F;-t) \Big)t^{d-1}\\
 &\ + \Big( e_{d-2}(P;-t) - \hspace{-.45cm}\sum_{\substack{F \subset P\\ \dim(F) = d-1}}\hspace{-.4cm}\frac{1}{2}e_{d-2}(F;-t) + \hspace{-.4cm}\sum_{\substack{G \subset P\\ \dim(G) = d-2}}\hspace{-.4cm}\omega_P(G)e_{d-2}(G;-t) \Big)t^{d-2}
 + \dots
\end{align*}

Great care has to be taken before using Theorems~\ref{thm:ed1-formula} and~\ref{thm:ed2-formula}, because these theorems assume the polytopes to be full-dimensional, while to use the expressions above we must consider all of the lower dimensional faces $F$. 

The main difference is that when $0 \notin \aff(F)$, then $L_F(t) = 0$ for all $t$ such that $\aff(tF)$ has no integer points. Letting $\bar{x}_F$ be the projection of $\aff(F)$ onto $\lin(F)^\perp$, we may express this condition equivalently as $t\bar{x}_F \notin \Lambda_F^* = \Proj_{\lin(F)^\perp}(\Z^d)$. Therefore we have to multiply the formulas from Theorems~\ref{thm:ed1-formula} and~\ref{thm:ed2-formula} by $\bm{1}_{\Lambda_F^*}(t\bar{x}_F)$ to take into account this effect.

Thus for $a_{d-1}(P;t)$ we obtain:
\begin{align*}
a_{d-1}(P;t) &= -e_{d-1}(P;-t) + \hspace{-.3cm} \sum_{\substack{F \subset P\\ \dim(F) = d-1}}\hspace{-.3cm} \frac{1}{2} e_{d-1}(F;-t)\\
&= \sum_{\substack{F \subset P\\ \dim(F) = d-1}} \hspace{-.3cm} \vol^*(F) \Big( \overline{B}_1^+\big(-\innerp{v_F}{\bar{x}_F}t\big) + \frac{1}{2}\bm{1}_{\Lambda_F^*}(-t\bar{x}_F) \Big)\\
&= \sum_{\substack{F \subset P\\ \dim(F) = d-1}} \hspace{-.3cm} \vol^*(F) \Big( \overline{B}_1\big(-\innerp{v_F}{\bar{x}_F}t\big) \Big)\\
&= -\hspace{-.4cm}\sum_{\substack{F \subset P\\ \dim(F) = d-1}} \hspace{-.3cm} \vol^*(F) \overline{B}_1\big(\innerp{v_F}{\bar{x}_F}t\big),
\end{align*}
where we have used that $v_F$ is $\Lambda_F$-primitive and hence $t\bar{x}_F \in \Lambda_F^*$ exactly when $\innerp{v_F}{\bar{x}_F}t \in \Z$.

For $a_{d-2}(P;t)$ we obtain:
\begin{align*}
a_{d-2}&(P;t) - e_{d-2}(P;-t) 
= -\hspace{-.3cm} \sum_{\substack{F \subset P\\ \dim(F) = d-1}}\hspace{-.3cm}\frac{1}{2}e_{d-2}(F;-t) + \hspace{-.4cm}\sum_{\substack{G \subset P\\ \dim(G) = d-2}} \hspace{-.4cm}\omega_P(G)e_{d-2}(G;-t)\\
&= \sum_{\substack{G \subset P\\ \dim(F) = d-2}} \hspace{-.3cm} \vol^*(G)\Big[ \bm{1}_{\Lambda_{G}^*}(t\bar{x}_G)\omega_P(G) + \frac{1}{2}
\bm{1}_{\Lambda_{F_1}^*}(t\bar{x}_{F_1})\overline{B}_1^+\Big( \Big\langle \frac{v_{F_1,G}}{\|v_{F_1,G}\|^2}, -t\bar{x}_G \Big\rangle \Big)\\ 
&\qquad+ \frac{1}{2} \bm{1}_{\Lambda_{F_2}^*}(t\bar{x}_{F_2})\overline{B}_1^+\Big( \Big\langle \frac{v_{F_2,G}}{\|v_{F_2,G}\|^2}, -t\bar{x}_G \Big\rangle \Big) \Big],
\end{align*}
where we have used that each codimension two face is a facet of exactly two codimension one faces and $v_{F_i,G}/\|v_{F_i,G}\|^2$ is the $\Lambda_G$-primitive vector in the direction of~$N_{F_i}(G)$. 

Next, since $\aff(G) \subset \aff(F)$, we may take $x_G \in \aff(G)$ as a representative of both $\aff(F_1)$ and $\aff(F_2)$. Using the expression $\bar{x}_G = x_1 v_{F_1, G} + x_2 v_{F_2, G}$ together with $\innerp{v_{F_1}}{t\bar{x}_G} = tkx_2$ (see the proof of Lemma~\ref{lem:hk}) we conclude that $t\bar{x}_{F_1} \in \Lambda_{F_1}^*$ if and only if $tkx_2 \in \Z$. Similarly, $t\bar{x}_{F_2} \in \Lambda_{F_2}^*$ if and only if $tkx_1 \in \Z$, so:
\begin{align*}
a_{d-2}&(P;t) - e_{d-2}(P;-t)\\
&= \sum_{\substack{G \subset P\\ \dim(F) = d-2}}\hspace{-.3cm}\vol^*(G)\Big[ \bm{1}_{\Lambda_{G}^*}(t\bar{x}_G)\omega_P(G) + \frac{1}{2} \bm{1}_{\Z}(kx_2t)\overline{B}_1^+\Big( \Big\langle \frac{v_{F_1,G}}{\|v_{F_1,G}\|^2}, -t\bar{x}_G \Big\rangle \Big)\\ 
&\qquad+ \frac{1}{2} \bm{1}_{\Z}(kx_1t)\overline{B}_1^+\Big( \Big\langle \frac{v_{F_2,G}}{\|v_{F_2,G}\|^2}, -t\bar{x}_G \Big\rangle \Big) \Big].
\end{align*}

Using $v_{F_2, G} = hv_{F_1,G} + kv_2$ and since $v_{F_1,G} / \|v_{F_1,G}\|^2 \in \Lambda_G$ and $v_2 \in \Lambda_G^*$, when~$kx_2t \in \Z$,
\begin{multline*}
\Big\langle \frac{v_{F_1,G}}{\|v_{F_1,G}\|^2}, -t\bar{x}_G \Big\rangle = -tx_1 -\frac{tx_2}{\|v_{F_1,G}\|^2} \innerp{v_{F_1,G}}{v_{F_2,G}}\\ 
= -tx_1 -thx_2 -tkx_2 \frac{\innerp{v_{F_1,G}}{v_2}}{\|v_{F_1,G}\|^2} = -tx_1 -thx_2\ (\hspace{-.33cm}\mod 1).
\end{multline*}
Similarly, let $h^{-1}$ be an integer such that $hh^{-1} = 1\  (\hspace{-.2cm}\mod k)$. Using $hv_{F_1,G} = v_{F_2,G} - kv_2$, $h^{-1} \in \Z$, $v_{F_2,G} / \|v_{F_2,G}\|^2 \in \Lambda_G$, and $v_2 \in \Lambda_G^*$, when $kx_1t \in \Z$, we get
\begin{multline*}
\Big\langle \frac{v_{F_2,G}}{\|v_{F_2,G}\|^2}, -t\bar{x}_G \Big\rangle = -tx_2 -tx_1 \Big\langle \frac{v_{F_2,G}}{\|v_{F_2,G}\|^2}, v_{F_1,G} \Big\rangle \\ 
= -tx_2 -tx_1h^{-1} \Big( 1 - k \Big\langle \frac{v_{F_2,G}}{\|v_{F_2,G}\|^2}, v_2 \Big\rangle \Big)\ (\hspace{-.33cm}\mod 1) = -t(h^{-1}x_1 + x_2)\  (\hspace{-.33cm}\mod 1).
\end{multline*}
Applying these relations to the main expression,
\begin{align*}
a_{d-2}&(P;t) - e_{d-2}(P;-t)\\
&= \sum_{\substack{G \subset P\\ \dim(G) = d-2}}\hspace{-.3cm}\vol^*(G)\Big[ \bm{1}_{\Lambda_{G}^*}(t\bar{x}_G)\omega_P(G) + \frac{1}{2} \bm{1}_{\Z}(kx_2t)\overline{B}_1^+\big(-(x_1 + hx_2)t \big)\\ 
&\qquad- \frac{1}{2} \bm{1}_{\Z}(kx_1t)\overline{B}_1\big((h^{-1}x_1 + x_2)t\big) - \frac{1}{4} \bm{1}_{\Z}(kx_1t)\bm{1}_{\Z}\big((h^{-1}x_1 + x_2)t\big) \Big].
\end{align*}

Next we use $\bm{1}_{\Z}(kx_1t)\bm{1}_{\Z}\big((h^{-1}x_1 + x_2)t\big) = \bm{1}_{\Lambda_{G}^*}(t\bar{x}_G)$. To see why this is true, from $\bar{x}_G = x_1 v_{F_1,G} + x_2 v_{F_2,G} = (x_1 + hx_2)v_1 + kx_2v_2$, we see that $\bar{x}_G \in \Lambda_G^*$ if and only if $kx_2 \in \Z$ and $x_1 + hx_2 \in \Z$. Hence, if $kx_1 \in \Z$ and $h^{-1}x_1 + x_2 \in \Z$, multiplying the second item by $h$ we conclude that $x_1 + hx_2 \in \Z$ while multiplying it by $k$ gives $h^{-1}kx_1 + kx_2 \in \Z$, so $kx_2 \in \Z$ and then $\bar{x}_G \in \Lambda_G^*$. The other direction is also easy.

Returning to the main expression,
\begin{align*}
a_{d-2}&(P;t) - e_{d-2}(P;-t)\\
&= \sum_{\substack{G \subset P\\ \dim(G) = d-2}}\hspace{-.3cm}\vol^*(G)\Big[ \Big(\omega_P(G) - \frac{1}{4}\Big) \bm{1}_{\Lambda_{G}^*}(t\bar{x}_G)\\
&\qquad+ \frac{1}{2} \bm{1}_{\Z}(kx_2t)\overline{B}_1^+\big(-(x_1 + hx_2)t \big) - \frac{1}{2} \bm{1}_{\Z}(kx_1t)\overline{B}_1\big((h^{-1}x_1 + x_2)t\big) \Big].
\end{align*}

To finish the verification of formula $a_{d-2}(P;t)$ from Theorem~\ref{thm:ad2-formula}, we must also take into account $e_{d-2}(P;-t)$ using the formula from Theorem~\ref{thm:ed2-formula}. For that, notice that the functions $\overline{B}_2$ are even while the Dedekind-Rademacher sum satisfies $s(h,k;-x,-y) = s(h,k;x,y)$. The other two terms with $\overline{B}_1$ cancels exactly the terms we got from the computation above, which completes the verification of Theorem~\ref{thm:ad2-formula}, given Theorem~\ref{thm:ed2-formula}.

\section{Proofs of two lemmas about lattices}

Here we prove the two lemmas from Section~\ref{sec:lattices}.

\detLperp*

\begin{proof}
Through this proof, for any set $v_1, \dots, v_s$ of vectors in $\R^d$, we use the notation $\det(v_1,\dots,v_s) := \det(V^{\sf T}V)^{1/2}$, where $V$ is the matrix with $v_1, \dots, v_s$ as columns.

Let $a_1, \dots, a_k$ be a basis for $L$ and $a_{k+1}, \dots, a_d$ be a completion to a basis for $\Lambda$ (that is possible since $L$ is primitive), so $\det(a_1,\dots,a_k) = \det(L)$ and $\det(a_1,\dots,a_d) = \det(\Lambda)$. Let $f_1, \dots, f_d$ be the dual basis for $\Lambda^*$, that is, $f_1, \dots, f_d$ are defined such that $\innerp{f_i}{a_j} = \delta_{i,j}$ for all $i, j = 1, \dots, d$. Note that $f_{k+1}, \dots, f_d$ is a basis for $L^\perp$, so $\det(f_{k+1}, \dots, f_d) = \det(L^\perp)$. 
 
Now, for $i = 1, \dots, k$, let $\tilde{f}_i := f_i - \Proj_{\spann(L)^\perp}(f_i)$, so that $\tilde{f}_i \in \spann(L)$ and $f_i - \tilde{f}_i \in \spann(L)^\perp = \spann(f_{k+1}, \dots, f_d)$. Since, for $i = 1, \dots, k$, the difference between $f_i$ and $\tilde{f}_i$ is a linear combination of $f_{k+1}, \dots, f_d$, we have that $\det(f_1, \dots, f_d) = \det(\tilde{f}_1, \dots, \tilde{f}_k, f_{k+1}, \dots, f_d)$ and since $\tilde{f}_1, \dots, \tilde{f}_k \in \spann(L)$ and $f_{k+1}, \dots, f_d \in \spann(L)^\perp$, we also have that 
 \begin{equation}\label{eq:det-prod}
\det(\tilde{f}_1, \dots, \tilde{f}_k, f_{k+1}, \dots, f_d) = \det(\tilde{f}_1, \dots, \tilde{f}_k) \det(f_{k+1}, \dots, f_d).  
 \end{equation}

Furthermore, since for all $i, j = 1, \dots, k$, $\langle \tilde{f_i}, a_j\rangle = \innerp{f_i}{a_j} = \delta_{i,j}$ and $\tilde{f_1}, \dots,$ $\tilde{f_k} \in \spann(L)$, they form a basis for $L^*$ and so $\det(\tilde{f}_1, \dots, \tilde{f}_k) = 1/\det(L)$.
 
Thus, from~\eqref{eq:det-prod}, we see that $1/\det(\Lambda) = \det(L^\perp)/\det(L)$, as desired.
\end{proof}

\duallattice*

\begin{proof}
As in the proof of the previous lemma, let $a_1, \dots, a_k$ be a basis for $L$, $a_{k+1}, \dots, a_d$ be a completion to a basis for $\Lambda$, and let $f_1, \dots, f_d$ be the dual basis for $\Lambda^*$, that is, $f_1, \dots, f_d$ are defined such that $\innerp{f_i}{a_j} = \delta_{i,j}$ for all $i, j = 1, \dots, d$. Denoting by $A_k$ the matrix with $a_1, \dots, a_k$ as columns, we have that $P = A_k(A_k^{\sf T}A_k)^{-1}A_k^{\sf T}$ is the orthogonal projection onto $\spann(L)$, indeed, $PA_k = A_k$ and $Pv = 0$ for $v \in \spann(L)^\perp$. Denoting by $F$ the matrix with $f_1, \dots, f_d$ as columns, we get that $\Pr_{\spann(L)}(\Lambda^*)$ is spanned by the columns of $PF = A_k(A_k^{\sf T}A_k)^{-1}A_k^{\sf T}F = \lp\begin{matrix} A_k(A_k^{\sf T}A_k)^{-1} \mid 0\end{matrix}\rp$. We finish the proof noting that the columns of $A_k(A_k^{\sf T}A_k)^{-1}$ are indeed a lattice basis for $L^*$, to see this simply note that $A_k^{\sf T}\lp A_k(A_k^{\sf T}A_k)^{-1}\rp = I$.
\end{proof}

\end{document}